\documentclass{scrartcl}
\usepackage{amssymb}
\usepackage{amsfonts}
\usepackage{amsmath}
\usepackage{pgf}
\usepackage{hyperref}
\newcommand{\bbbc}{\mathbb{C}}
\newcommand{\bbbn}{\mathbb{N}}
\newcommand{\bbbr}{\mathbb{R}}
\newcommand{\Idx}{\mathcal{I}}
\newcommand{\Jdx}{\mathcal{J}}
\newcommand{\ctI}{\mathcal{T}_{\Idx}}
\newcommand{\lfI}{\mathcal{L}_{\Idx}}
\newcommand{\ctIl}[1]{\mathcal{T}_{\Idx}^{(#1)}}
\newcommand{\ctJ}{\mathcal{T}_{\Jdx}}
\newcommand{\ctIJ}{\mathcal{T}_{\Idx\times\Jdx}}
\newcommand{\ctII}{\mathcal{T}_{\Idx\times\Idx}}
\newcommand{\lfII}{\mathcal{L}_{\Idx\times\Idx}}
\newcommand{\lfiII}{\mathcal{L}_{\Idx\times\Idx}^-}
\newcommand{\lfaII}{\mathcal{L}_{\Idx\times\Idx}^+}
\newcommand{\dist}{\mathop{\operatorname{dist}}\nolimits}
\newcommand{\diam}{\mathop{\operatorname{diam}}\nolimits}
\newcommand{\sons}{\mathop{\operatorname{sons}}\nolimits}
\newcommand{\treeroot}{\mathop{\operatorname{root}}\nolimits}
\newcommand{\supp}{\mathop{\operatorname{supp}}\nolimits}
\newcommand{\level}{\mathop{\operatorname{level}}\nolimits}
\newcommand{\brow}{\mathop{\operatorname{row}}\nolimits}
\newcommand{\bcol}{\mathop{\operatorname{col}}\nolimits}
\newcommand{\desc}{\mathop{\operatorname{desc}}\nolimits}

\newcommand{\sd}[1]{\mathop{\operatorname{sd}}_{#1}\nolimits}
\newcommand{\descd}[1]{\mathop{\operatorname{dscd}}_{#1}\nolimits}
\newcommand{\Csb}{C_\text{sb}}
\newcommand{\Csn}{C_\text{sn}}
\newcommand{\Cbp}{C_\text{bp}}
\newcommand{\Cbb}{C_\text{bb}}
\newcommand{\Cov}{C_\text{ov}}
\newcommand{\Crs}{C_\text{rs}}
\newcommand{\Cun}{C_\text{un}}
\newcommand{\Csp}{C_\text{sp}}
\newcommand{\Clv}{C_\text{lv}}
\newcommand{\Ccb}{C_\text{cb}}
\newcommand{\Cnc}{C_\text{nc}}
\newcommand{\Clb}{C_\text{lb}}
\newcommand{\Cdh}{C_\text{dh}}
\newcommand{\Csvd}{C_\text{svd}}
\newcommand{\Cba}{C_\text{ba}}
\newtheorem{theorem}{Theorem}
\newtheorem{lemma}[theorem]{Lemma}

\newtheorem{definition}[theorem]{Definition}
\newtheorem{remark}[theorem]{Remark}

\newenvironment{proof}{\emph{Proof.}}{\hfill$\Box$}
\allowdisplaybreaks
\title{Directional $\mathcal{H}^2$-matrix compression for
       high-frequency problems}
\author{Steffen B\"orm}
\date{16th of September, 2015}
\begin{document}
\maketitle

\begin{abstract}
Standard numerical algorithms like the fast multipole method or
$\mathcal{H}$-matrix schemes rely on low-rank approximations of
the underlying kernel function.
For high-frequency problems, the ranks grow rapidly as the mesh
is refined, and standard techniques are no longer attractive.

\emph{Directional} compression techniques solve this problem
by using decompositions based on plane waves.
Taking advantage of hierarchical relations between these waves'
directions, an efficient approximation is obtained.

This paper is dedicated to \emph{directional $\mathcal{H}^2$-matrices}
that employ local low-rank approximations to handle directional
representations efficiently.

The key result is an algorithm that takes an arbitrary matrix and
finds a quasi-optimal approximation of this matrix as a directional
$\mathcal{H}^2$-matrix using a prescribed block tree.
The algorithm can reach any given accuracy, and the approximation
requires only $\mathcal{O}(n k + \kappa^2 k^2 \log n)$ units of
storage, where $n$ is the matrix dimension, $\kappa$ is the
wave number, and $k$ is the local rank.
In particular, we have a complexity of $\mathcal{O}(n k)$ if $\kappa$
is constant and $\mathcal{O}(n k^2 \log n)$ for high-frequency problems
characterized by $\kappa^2 \sim n$.

Since the algorithm can be applied to arbitrary matrices, it can serve
as the foundation of fast techniques for constructing preconditioners.
\end{abstract}

% ============================================================
%
% Introduction
%
% ============================================================
\section{Introduction}

We consider the Helmholtz single layer potential operator
\begin{equation*}
  \mathcal{G}[u](x) := \int_\Omega g(x,y) u(y) \,dy,
\end{equation*}
where $\Omega\subseteq\bbbr^3$ is a surface and
\begin{equation}\label{eq:helmholtz}
  g(x,y) = \frac{\exp(i \kappa \|x-y\|)}{4\pi \|x-y\|}
\end{equation}
denotes the Helmholtz kernel function with the
wave number $\kappa\in\bbbr_{\geq 0}$.

Applying a standard Galerkin discretization scheme with
a finite element basis $(\varphi_i)_{i\in\Idx}$ leads to the
stiffness matrix $G\in\bbbc^{\Idx\times\Idx}$ given by
\begin{align}\label{eq:matrix}
  g_{ij}
  &= \int_\Omega \varphi_i(x) \int_\Omega g(x,y) \varphi_j(y) \,dy\,dx &
  &\text{ for all } i,j\in\Idx,
\end{align}
where we assume that the basis functions are sufficiently smooth
to ensure that the integrals are well-defined even for $x=y$.
Due to $g(x,y)\neq 0$ for all $x\neq y$, the matrix $G$ is not
sparse and therefore requires special handling if we want to
construct an efficient algorithm.

Standard techniques like fast multipole expansions \cite{RO85,GRRO87},
panel clustering \cite{HANO89,SA00}, or hierarchical matrices
\cite{HA99,HAKH00,GRHA02} rely on local low-rank approximations
of the matrix.

In the case of the high-frequency Helmholtz equation, e.g., if
the product of the wave number $\kappa$ and the mesh width $h$
is relatively large, these techniques can no longer be applied
since the local ranks become too large.

The \emph{fast multipole method} can be generalized to handle this
problem by employing a special expansion that leads to operators
that can be diagonalized, and therefore evaluated efficiently
\cite{RO93,GRHUROWA98}.

The \emph{butterfly method} (also known as multi-level matrix
decomposition algorithms, MLMDA) \cite{BOMI96} achieves a similar
goal by using permutations and block-diagonal transformations in
a pattern closely related to the fast Fourier transformation
algorithm.

\emph{Directional methods} \cite{BR91,ENYI07,MESCDA12,BEKUVE15}
take advantage of the fact that the Helmholtz kernel (\ref{eq:helmholtz})
can be written as a product of a plane wave and a function that
is smooth inside a conical domain.
Replacing this smooth function by a suitable approximation results
in fast summation schemes.

We should also mention that there are specialized
methods for certain geometries:
modified Ewalt summation methods can be applied for periodic
geometries \cite{BEKUMO08},
while tensor-structured meshes have properties that allow us to
construct efficient algorithms even though standard
Chebyshev interpolation \cite{KH09} requires polynomials of a fairly
high degree.

We will focus on directional methods, since they can be applied
in a more general setting than the fast multipole expansions based
on special functions and since they offer the chance of achieving
better compression rates than the butterfly scheme.

In particular, we will work with \emph{directional $\mathcal{H}^2$-matrices}
(abbreviated $\mathcal{DH}^2$-matrices) introduced in \cite{BEKUVE15,BOME15},
the algebraic counterparts of the directional approximation schemes used
in \cite{BR91,ENYI07,MESCDA12}.
The article \cite{BEKUVE15} presents a cross approximation scheme
that can be used to construct a $\mathcal{DH}^2$-matrix approximation
based on a relatively small number of matrix coefficients.

In the present paper, we follow a different approach:
there is an algorithm \cite{BOHA02} that approximates an arbitrary
matrix by an $\mathcal{H}^2$-matrix, guaranteeing a prescribed accuracy.
Relying exclusively on orthogonal transformations, the algorithm
is very stable and has shown itself to be very reliable in practice.

Although the algorithm is originally formulated for dense matrices,
it can be easily adapted to compress $\mathcal{H}$- and
$\mathcal{H}^2$-matrices \cite{HAKHSA00,BO04}, to merge
$\mathcal{H}^2$-submatrices \cite{BO07a}, and to perform matrix
arithmetic operations like addition, multiplication and inversion
\cite{BORE14} to construct robust data-sparse preconditioners.

We will generalize this fundamental algorithm to construct
$\mathcal{DH}^2$-matrices, present rigorous error estimates,
outline an efficient error control strategy, and provide bounds for
the storage requirements and computational work.

In the original paper \cite{BEKUVE15} a directional approximation
is used only for large subdomains, while a standard $\mathcal{H}$-matrix
approximation is used for small subdomains.
Due to the structure of $\mathcal{H}$-matrices, this approach cannot
reach linear complexity even in the low-frequency case.
The rank $k$ depends on the desired accuracy $\epsilon$,
typically like $k\sim|\log(\epsilon)|^\alpha$ with a small exponent $\alpha$.

In the present paper, we generalize the more efficient $\mathcal{H}^2$-matrix
representations and obtain a complexity of
$\mathcal{O}(n k + \kappa^2 k^2 \log(n))$, i.e., we have
$\mathcal{O}(n k)$ for low-frequency problems and
$\mathcal{O}(n k^2 \log(n))$ in the case of high frequencies
where $\kappa h\sim 1$ holds for the mesh parameter $h\sim n^{-1/2}$.
It is possible to prove \cite{BOME15} that
a rank of $k\sim|\log(\epsilon)|^3$ is sufficient, independent of
$\kappa$ and $n$.
Not only does this unified approach improve the complexity estimates,
it also allows us to use the same algorithms across the entire
range of resolutions and frequencies and leads to ``cleaner''
implementations.

The paper is organized as follows:
Section~2 introduces $\mathcal{DH}^2$-matrices, using directional
interpolation methods \cite{MESCDA12} as an example motivating the
algebraic structure.

Section~3 presents admissibility conditions that allow us to find
submatrices that can be approximated by low rank.
By choosing the required directions in a separate step, we can keep
the algorithm quite close to standard techniques for $\mathcal{H}$-
and $\mathcal{H}^2$-matrices.

Section~4 outlines how the standard $\mathcal{H}^2$-matrix-vector
multiplication algorithm can be generalized to handle
directional approximations, in particular how forward and
backward transformations have to be modified.

Section~5 is devoted to the first main result, an estimate for
the storage requirements and the computational work of the
matrix-vector multiplication.
Compared to previous work, we use significantly weaker assumptions:
the clusters are not required to shrink at a fixed rate, the leaf
clusters are not required to be on the same level of the tree,
and we can handle volumes, surfaces and curves uniformly via
the ``curvature condition'' (\ref{eq:planar}).
The main result in this section is the complexity estimate
provided by Theorem~\ref{th:complexity}.

Section~6 introduces the fundamental compression algorithm
that can be used to approximate an arbitrary matrix by
an $\mathcal{DH}^2$-matrix guaranteeing a prescribed accuracy.
The main results in this section are Theorem~\ref{th:error_estimate},
that gives an estimate for the approximation error that depends only
on quantities we can control explicitly, and
Theorem~\ref{th:complexity_compression} that gives an estimate
for the computational work required by
the compression algorithm, which is close to optimal given the
amount of data that has to be processed.
Techniques like $\mathcal{H}$-matrix condensation
\cite[Chapter~6.5]{BO10} and weighted truncation \cite[Chapter~6.6]{BO10}
that have been established in the context of $\mathcal{H}^2$-matrix
methods to improve the efficiency can easily be extended to handle
$\mathcal{DH}^2$-matrices.

The final Section~7 contains a number of numerical experiments that
illustrate the potential of the $\mathcal{DH}^2$-matrix representation
and compare the new class of matrices with the
well-known ACA scheme.

% ============================================================
%
% Directional H^2-matrices
%
% ============================================================
\section{\texorpdfstring{Directional $\mathcal{H}^2$-matrices}
                        {Directional H2-matrices}}

Hierarchical matrix methods are based on decompositions of the
matrix $G$ into submatrices that can be approximated by factorized
low-rank matrices.
In order to find out how these submatrices and the factorization
have to be chosen, we consider the approximation scheme described
in \cite{MESCDA12} and translate the resulting compressed
representation into an algebraical definition that can be applied
in more general situations.

In order to describe the decomposition into submatrices, we first
introduce a hierarchy of subsets of the index set $\Idx$ corresponding
to the \emph{box trees} used, e.g., in fast multipole methods.

%
% Definition: Cluster tree
%
\begin{definition}[Cluster tree]
\label{de:cluster}
Let $\mathcal{T}$ be a labeled tree such that the label $\hat t$
of each node $t\in\mathcal{T}$ is a subset of the index set $\Idx$.
We call $\mathcal{T}$ a \emph{cluster tree} for $\Idx$ if
\begin{itemize}
  \item the root $r\in\mathcal{T}$ is assigned $\hat r=\Idx$,
  \item the index sets of siblings are disjoint, i.e.,
    \begin{align*}
      t_1\neq t_2 &\Longrightarrow \hat t_1\cap\hat t_2=\emptyset &
      &\text{ for all } t\in\mathcal{T},\ t_1,t_2\in\sons(t),
         \text{ and}
    \end{align*}
  \item the index sets of a cluster's sons are a partition of
    their father's index set, i.e.,
    \begin{align*}
      \hat t &= \bigcup_{t'\in\sons(t)} \hat t' &
      &\text{ for all } t\in\mathcal{T}
          \text{ with } \sons(t)\neq\emptyset.
    \end{align*}
\end{itemize}
A cluster tree for $\Idx$ is usually denoted by $\ctI$.
Its nodes are called \emph{clusters}, and
its root is denoted by $\treeroot(\ctI)$.
\end{definition}

A cluster tree $\ctI$ can be split into levels:
we let $\ctIl{0}$ be the set containing only the root of $\ctI$
and define
\begin{align*}
  \ctIl{\ell} &:= \{ t'\in\ctI\ :\ t'\in\sons(t) \text{ for a }
                       t\in\ctIl{\ell-1} \} &
  &\text{ for all } \ell\in\bbbn.
\end{align*}
For each cluster $t\in\ctI$, there is exactly one $\ell\in\bbbn_0$
such that $t\in\ctIl{\ell}$.
We call this the \emph{level number} of $t$ and denote it by $\level(t)=\ell$.
The maximal level
\begin{equation*}
  p_\Idx := \max\{ \level(t)\ :\ t\in\ctI \}
\end{equation*}
is called the \emph{depth} of the cluster tree.

Pairs of clusters $(t,s)$ correspond to subsets $\hat t\times\hat s$
of $\Idx\times\Idx$, i.e., to submatrices of $G\in\bbbc^{\Idx\times\Idx}$.
These pairs inherit the hierarchical structure provided by the
cluster tree.

%
% Definition: Block tree
%
\begin{definition}[Block tree]
\label{de:block}
Let $\mathcal{T}$ be a labeled tree, and let $\ctI$ and $\ctJ$ be
cluster trees for index sets $\Idx$ and $\Jdx$ with roots $r_\Idx$
and $r_\Jdx$.
We call $\mathcal{T}$ a \emph{block tree} for $\ctI$ and $\ctJ$ if
\begin{itemize}
  \item for each $b\in\mathcal{T}$ there are $t\in\ctI$, $s\in\ctJ$
        such that $b=(t,s)$,
  \item the root $r\in\mathcal{T}$ satisfies $r=(r_\Idx,r_\Jdx)$,
  \item the label of $b=(t,s)\in\mathcal{T}$ is given by
        $\hat b=\hat t\times\hat s$, and
  \item for each $b=(t,s)\in\mathcal{T}$ we have
    \begin{equation*}
      \sons(b)\neq\emptyset \Longrightarrow
      \sons(b) = \sons(t)\times\sons(s).
    \end{equation*}
\end{itemize}
A block tree for $\ctI$ and $\ctJ$ is usually denoted by $\ctIJ$.
Its nodes are called \emph{blocks}.
\end{definition}

In the following, we assume that a cluster tree $\ctI$ for the
index set $\Idx$ and a block tree $\ctII$ for $\ctI$ are given.

We have to identify submatrices, corresponding to blocks, that can
be approximated efficiently.
Considering a submatrix $G|_{\hat t\times\hat s}$ corresponding to
a block $b=(t,s)$, the standard approach is to use a tensor-product
approximation of the kernel function $g$ that is valid on the
supports of the basis functions $\varphi_i$ and $\varphi_j$ for
$i\in\hat t$ and $j\in\hat s$, respectively.
Since working directly with these supports is too expensive, we
introduce \emph{bounding boxes}, i.e., for each $t\in\ctI$
we construct an axis-parallel box $B_t$ such that
\begin{align*}
  \supp(\varphi_i) &\subseteq B_t &
  &\text{ for all } i\in\hat t.
\end{align*}
If we have a tensor-product approximation $\tilde g_{ts}$ of $g$ that
is sufficiently accurate on $B_t\times B_s$, we can use it to replace
$g$ in (\ref{eq:matrix}) and obtain a low-rank matrix.

Standard techniques for finding such local approximations of $g$
include fast multipole expansions \cite{RO85,GRRO87}, panel-clustering
techniques \cite{HANO89,SA00}, cross-approximation methods
\cite{BOGR04}, and recently quadrature formulas \cite{BOCH14}.

In order to handle the oscillatory nature of the Helmholtz kernel
(\ref{eq:helmholtz}), we focus on directional approximation methods
\cite{BR91,ENYI07,MESCDA12,BEKUVE15} that are based on the idea
that plane waves
\begin{equation*}
  x \mapsto \exp(i \kappa \langle x-y, c \rangle)
\end{equation*}
can be used to approximate the spherical waves appearing in the
kernel function.
Here $\kappa$ is again the wave number, while $c\in\bbbr^3$ is a
unit vector describing the direction in which the wave is traveling.

%
% Definition: Hierarchical directions
%
\begin{definition}[Hierarchical directions]
\label{de:directions}
Let $(\mathcal{D}_\ell)_{\ell=0}^\infty$ be a family of finite subsets
of $\bbbr^3$.
It is called a \emph{family of hierarchical directions} if
\begin{align*}
  \|c\|=1 &\vee c=0 &
  &\text{ for all } c\in\mathcal{D}_\ell,\ \ell\in\bbbn_0.
\end{align*}
Let $(\sd{\ell})_{\ell=0}^\infty$ be a family of mappings
$\sd{\ell}:\mathcal{D}_\ell\to\mathcal{D}_{\ell+1}$.
It is called a \emph{family of compatible son mappings} if
\begin{align*}
  \|c - \sd{\ell}(c)\| &\leq \|c-\tilde c\| &
  &\text{ for all } c\in\mathcal{D}_\ell,\ \tilde c\in\mathcal{D}_{\ell+1},
                     \ \ell\in\bbbn_0.
\end{align*}
Given a cluster tree $\ctI$, a family of hierarchical directions and
a family of compatible son mappings, we write
\begin{align*}
  \mathcal{D}_t &:= \mathcal{D}_{\level(t)}, &
  \sd{t}(c) &:= \sd{\level(t)}(c) &
  &\text{ for all } t\in\ctI,\ c\in\mathcal{D}_{\level(t)}.
\end{align*}
\end{definition}

Let $(\mathcal{D}_\ell)_{\ell=0}^\infty$  be a family of hierarchical
directions and $(\sd{\ell})_{\ell=0}^\infty$ a family of compatible son
mappings.

We consider the approximation of the submatrix $G|_{\hat t\times\hat s}$
corresponding to a block $b=(t,s)\in\ctII$.
By Definition~\ref{de:block}, we know that $t$ and $s$ are clusters
on the same level of $\ctI$, and Definition~\ref{de:directions}
implies $\mathcal{D}_t=\mathcal{D}_s$.

For a given direction $c\in\mathcal{D}_t=\mathcal{D}_s$, we find
\begin{equation}\label{eq:helmholtz_c}
  g(x,y) = \frac{\exp(i \kappa \|x-y\|)}{\|x-y\|}
         = \underbrace{\frac{\exp(i \kappa \|x-y\|
                                  - i \kappa \langle x-y, c \rangle)}
                            {\|x-y\|}}_{=:g_c(x,y)}
           \exp(i \kappa \langle x-y, c \rangle).
\end{equation}
Due to
\begin{align*}
  \kappa \|x-y\| - \kappa \langle x-y, c \rangle
  &= \kappa \|x-y\| (1 - \cos\angle(x-y, c))\\
  &\approx \frac{\kappa}{2} \|x-y\| \sin^2\angle(x-y,c),
\end{align*}
the function
\begin{equation*}
  g_c(x,y) = \frac{\exp(i \kappa \|x-y\| - i \kappa \langle x-y, c \rangle)}
                  {\|x-y\|}
\end{equation*}
is smooth as long as the angle between $x-y$ and $c$ is sufficiently
small and we are sufficiently far from the singularity.

Following \cite{MESCDA12}, we can approximate $g_c$ by interpolation.
We fix
\begin{itemize}
  \item interpolation points $(\xi_{t,\nu})_{\nu=1}^k$ for each
        cluster $t\in\ctI$ and
  \item corresponding Lagrange polynomials $(\ell_{t,\nu})_{\nu=1}^k$.
\end{itemize}
Tensor Chebyshev points are a good choice, since they lead to an
interpolation scheme of almost optimal stability.
The interpolating polynomial for $g_c$ is given by
\begin{equation*}
  \tilde g_{c,ts}(x,y)
  = \sum_{\nu=1}^k \sum_{\mu=1}^k g_c(\xi_{t,\nu},\xi_{s,\mu})
             \ell_{t,\nu}(x) \ell_{s,\mu}(y),
\end{equation*}
and using it to replace $g_c$ in (\ref{eq:helmholtz_c}) leads to
\begin{align*}
  g(x,y) &= g_c(x,y) \exp(i \kappa \langle x-y, c \rangle)
   \approx \tilde g_{c,ts}(x,y) \exp(i \kappa \langle x-y, c \rangle)\\
  &= \sum_{\nu=1}^k \sum_{\mu=1}^k
      g_c(\xi_{t,\nu}, \xi_{s,\mu})
      \exp(i \kappa \langle x, c \rangle) \ell_{t,\nu}(x)
      \exp(-i \kappa \langle y, c \rangle) \ell_{s,\mu}(y).
\end{align*}
We introduce
\begin{gather*}
  s_{b,\nu\mu} := g_c(\xi_{t,\nu},\xi_{s,\mu}),\\
  \ell_{tc,\nu}(x) := \exp(i \kappa \langle x, c \rangle) \ell_{t,\nu}(x),\qquad
  \ell_{sc,\mu}(x) := \exp(i \kappa \langle y, c \rangle) \ell_{s,\mu}(y),
\end{gather*}
and obtain
\begin{equation}\label{eq:kernel_apx}
  g(x,y) \approx \tilde g_b(x,y)
         := \sum_{\nu=1}^k \sum_{\mu=1}^k s_{b,\nu\mu}
                 \ell_{tc,\nu}(x) \overline{\ell_{sc,\mu}(y)}.
\end{equation}
Replacing $g$ by $\tilde g_b$ in (\ref{eq:matrix}) yields
\begin{align*}
  g_{ij} &\approx \int_\Omega \varphi_i(x) \int_\Omega
              \tilde g_b(x,y) \varphi_j(y) \,dy\,dx\\
  &= \sum_{\nu=1}^k \sum_{\mu=1}^k s_{b,\nu\mu}
         \underbrace{\int_\Omega \varphi_i(x)
                                \ell_{tc,\nu}(x) \,dx}_{=:v_{tc,i\nu}}
         \int_\Omega \varphi_j(y) \overline{\ell_{sc,\mu}(y)} \,dy\\
  &= \sum_{\nu=1}^k \sum_{\mu=1}^k s_{b,\nu\mu}
         v_{tc,i\nu} \overline{v_{sc,j\mu}}
   = (V_{tc} S_b V_{sc}^*)_{ij}
   \qquad\text{ for all } i\in\hat t,\ j\in\hat s.
\end{align*}
Due to $S_b\in\bbbc^{k\times k}$, this is a factorized low-rank approximation
\begin{equation}\label{eq:matrix_apx}
  G|_{\hat t\times\hat s} \approx V_{tc} S_b V_{sc}^*.
\end{equation}
In order to handle the matrices $V_{tc}$ efficiently, we take advantage
of the hierarchy provided by the cluster tree:
given a son $t'\in\sons(t)$ and the best approximation $c':=\sd{t}(c)$ of
the direction $c$ in $t'$, we look for a matrix $E_{t'c}\in\bbbc^{k\times k}$
such that
\begin{equation}\label{eq:nested_apx}
  V_{tc}|_{\hat t'\times k} \approx V_{t'c'} E_{t'c}.
\end{equation}
This property allows us to avoid storing $V_{tc}\in\bbbc^{\hat t\times k}$
by only storing the small matrices $E_{t'c}\in\bbbc^{k\times k}$.
We can construct the matrix $E_{t'c}$ by interpolating the function
\begin{equation*}
  x \mapsto \exp(i \kappa \langle x, c-c' \rangle) \ell_{t,\nu}(x).
\end{equation*}
Assuming that the angle between $c$ and $c'$ is sufficiently small,
this function is smooth and we find
\begin{align*}
  \ell_{tc,\nu}(x)
  &= \exp(i \kappa \langle x, c \rangle) \ell_{t,\nu}(x)
   = \exp(i \kappa \langle x, c' \rangle)
     \exp(i \kappa \langle x, c-c' \rangle) \ell_{t,\nu}(x)\\
  &\approx \exp(i \kappa \langle x, c' \rangle)
     \sum_{\nu'=1}^k
        \underbrace{\exp(i \kappa \langle \xi_{t',\nu'}, c-c' \rangle)
                    \ell_{t,\nu}(\xi_{t',\nu'})}_{=:e_{t'c,\nu'\nu}}
                    \ell_{t',\nu'}(x)\\
  &= \sum_{\nu'=1}^k e_{t'c,\nu'\nu} \ell_{t'c',\nu'}(x).
\end{align*}
This approach immediately yields
\begin{align*}
  v_{tc,i\nu}
  &= \int_\Omega \varphi_i(x) \ell_{tc,\nu}(x) \,dx
   \approx \sum_{\nu'=1}^k e_{t'c,\nu'\nu}
             \int_\Omega \varphi_i(x) \ell_{t'c',\nu'}(x) \,dx
   = (V_{t'c'} E_{t'c})_{i\nu}
\end{align*}
for all $i\in\hat t'$ and $\nu\in[1:k]$, which is equivalent
to (\ref{eq:nested_apx}).

The notation ``$E_{t'c}$'' (instead of something like ``$E_{t'tc'c}$''
listing all parameters) for the matrices is justified since the father
$t\in\ctI$ is uniquely determined by $t'\in\ctI$ due to the tree structure
and the direction $c'=\sd{t}(c)$ is uniquely determined by
$c\in\mathcal{D}_t$ due to our Definition~\ref{de:directions}.

Our (approximate) equation (\ref{eq:nested_apx}) gives rise to the following
definition.

%
% Definition: Directional cluster basis
%
\begin{definition}[Directional cluster basis]
Let $k\in\bbbn$, and let $V=(V_{tc})_{t\in\ctI,c\in\mathcal{D}_t}$ be a family
of matrices.
We call it a \emph{directional cluster basis} if
\begin{itemize}
  \item $V_{tc}\in\bbbc^{\hat t\times k}$ for all $t\in\ctI$
     and $c\in\mathcal{D}_t$, and
  \item there is a family
     $E=(E_{t'c})_{t\in\ctI,t'\in\sons(t),c\in\mathcal{D}_t}$
     such that
     \begin{align}\label{eq:nested}
       V_{tc}|_{\hat t'\times k} &= V_{t'c'} E_{t'c} &
       &\text{ for all } t\in\ctI,\ t'\in\sons(t),\ c\in\mathcal{D}_t,
                        \ c'=\sd{t}(c).
     \end{align}
\end{itemize}
The elements of the family $E$ are called \emph{transfer matrices} for
the directional cluster basis $V$, and $k$ is called its \emph{rank}.
\end{definition}

We can now define the class of matrices that is the subject of this
article:
we denote the \emph{leaves} of the block tree $\ctII$ by
\begin{equation*}
  \lfII := \{ b\in\ctII\ :\ \sons(b)=\emptyset \}
\end{equation*}
and have to represent each of the submatrices $G|_{\hat t\times\hat s}$
for $b=(t,s)\in\lfII$.
For most of these submatrices, we can find an approximation of the
form (\ref{eq:matrix_apx}).
These matrices are called \emph{admissible} and collected in a
subset
\begin{align*}
  \lfaII &:= \{ b\in\lfII\ :\ b \text{ is admissible} \}.\\
\intertext{The remaining blocks are called \emph{inadmissible} and
collected in the set}
  \lfiII &:= \lfII \setminus \lfaII.
\end{align*}
How to decide whether a block is admissible or not is the topic of the
next section.

%
% Definition: Directional H^2-matrix
%
\begin{definition}[Directional $\mathcal{H}^2$-matrix]
Let $V$ and $W$ be directional cluster bases for $\ctI$.
Let $G\in\bbbc^{\Idx\times\Idx}$ be matrix.
We call it a \emph{directional $\mathcal{H}^2$-matrix} or short an
\emph{$\mathcal{DH}^2$-matrix} if there are families
$S=(S_b)_{b\in\lfaII}$ and $(c_b)_{b\in\lfaII}$ such that
\begin{itemize}
  \item $S_b\in\bbbc^{k\times k}$ and $c_b\in\mathcal{D}_t=\mathcal{D}_s$
    for all $b=(t,s)\in\lfaII$, and
  \item $G|_{\hat t\times\hat s} = V_{tc} S_b W_{sc}^*$ with $c=c_b$ for all
    $b=(t,s)\in\lfaII$.
\end{itemize}
The elements of the family $S$ are called \emph{coupling matrices},
and $c_b$ is called the \emph{block direction} for $b\in\ctII$.
The cluster bases $V$ and $W$ are called the \emph{row cluster basis}
and \emph{column cluster basis}, respectively.

A \emph{$\mathcal{DH}^2$-matrix representation} of a
$\mathcal{DH}^2$-matrix $G$ consists of $V$, $W$, $S$ and the
family $(G|_{\hat b})_{b\in\lfiII}$ of \emph{nearfield matrices} corresponding
to the inadmissible leaves of $\ctII$.
\end{definition}

%
% Remark: Directional H^2-matrix
%
\begin{remark}[Directional $\mathcal{H}^2$-matrix]
The name ``directional $\mathcal{H}^2$-matrix'' is also used in
\cite{BEKUVE15}, but defined a little differently:
instead of using the representation (\ref{eq:matrix_apx}) for all
admissible blocks, the algorithm in \cite{BEKUVE15} uses
standard $\mathcal{H}$-matrices for small clusters.

In the case of a constant wave number $\kappa$, our approach
yields a complexity of $\mathcal{O}(n k)$ (cf. Theorem~\ref{th:complexity}),
while the approach in \cite{BEKUVE15} can be no better than
$\Omega(n k \log n)$.
\end{remark}

% ============================================================
%
% Admissibility
%
% ============================================================
\section{Admissibility}
\label{se:admissibility}

In order to construct a $\mathcal{DH}^2$-matrix approximation,
we have to find a cluster tree $\ctI$, a block tree $\ctII$,
and a family of hierarchical directions $(\mathcal{D}_t)_{t\in\ctI}$
such that
\begin{align*}
  G|_{\hat t\times\hat s}
  &\approx V_{tc} S_b W_{sc}^* &
  &\text{ for all } b=(t,s)\in\lfaII,\ c=c_b.
\end{align*}
An analysis of the approximation scheme (cf., e.g., \cite{MESCDA12})
indicates that this is the case if three \emph{admissibility
conditions} hold:
\begin{subequations}\label{eq:admissibility}
\begin{align}
  \kappa \left\| \frac{m_t-m_s}{\|m_t-m_s\|} - c \right\|
  &\leq \frac{\eta_1}{\max\{\diam(B_t),\diam(B_s)\}},
          \label{eq:adm_directional}\\
  \kappa \max\{ \diam^2(B_t), \diam^2(B_s) \}
  &\leq \eta_2 \dist(B_t, B_s),\text{ and}
          \label{eq:adm_parabolic}\\
  \max\{ \diam(B_t), \diam(B_s) \}
  &\leq \eta_2 \dist(B_t, B_s)
          \label{eq:adm_standard},
\end{align}
\end{subequations}
where $m_t$ and $m_s$ denote the centers of the bounding boxes
$B_t$ and $B_s$ and $\eta_1,\eta_2>0$ are parameters that can
be chosen to balance storage requirements and accuracy.

The first condition (\ref{eq:adm_directional}) ensures that
the direction $c$ of the plane-wave approximation is sufficiently
close to the direction of the wave traveling from $m_t$ to $m_s$.

The second condition (\ref{eq:adm_parabolic}) is equivalent to
\begin{equation*}
  \frac{\max\{\diam(B_t),\diam(B_s)\}}
       {\dist(B_t,B_s)}
  \leq \frac{\eta_2}{\kappa \max\{\diam(B_t), \diam(B_s)\}},
\end{equation*}
it ensures that the angle between all vectors $x-y$ for $x\in B_t$
and $y\in B_s$ is bounded and that this bound shrinks when the wave
number or the cluster diameter grows.

The third condition (\ref{eq:adm_standard}) provides an
upper bound for the same angle that is independent of
wave number and cluster diameter.
This is the standard admissibility condition that is also used
for the Laplace equation or linear elasticity.

If a block $b=(t,s)$ satisfies the admissibility conditions
(\ref{eq:admissibility}), it is possible to prove that
directional interpolation converges at an exponential
rate that depends only on the parameters $\eta_1$, $\eta_2$,
and the stability constant of the interpolation scheme
\cite{MESCDA12,BOME15}.

In order to obtain a simple algorithm, we treat the first
condition (\ref{eq:adm_directional}) separately from the others:
for each level $\ell$ of the cluster tree, we compute the maximal
diameter
\begin{equation*}
  \delta_\ell := \max\{ \diam(B_t)\ :\ t\in\ctI,\ \level(t)=\ell \}
\end{equation*}
of all bounding boxes and then fix a set of directions $\mathcal{D}_\ell$
such that
\begin{align*}
  \min\{ \| z - c \|\ :\ c\in\mathcal{D}_\ell \}
  &\leq \frac{\eta_1}{\kappa \delta_\ell} &
  &\text{ for all } z\in\bbbr^3,\ \|z\|=1.
\end{align*}
If all clusters on level $\ell$ share this set $\mathcal{D}_\ell$ of
directions, the condition (\ref{eq:adm_directional}) is guaranteed
since $(m_t-m_s)/\|m_t-m_s\|$ is a unit vector for all $t,s\in\ctI$.

%
% Figure: Projected directions
%
\begin{figure}
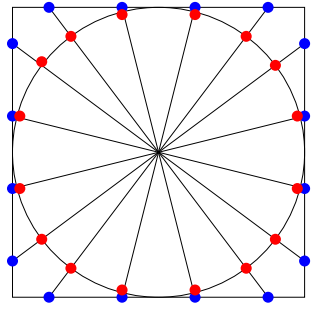

  \pgfdeclareimage[width=4cm]{directions}{fi_directions}

  \begin{center}
    \pgfuseimage{directions}
  \end{center}

  \caption{Constructing directions by projecting a regular mesh
           to the unit circle}
  \label{fi:projected_directions}
\end{figure}

In our numerical experiments, we construct the sets $\mathcal{D}_\ell$
by splitting the surface of the cube $[-1,1]^3$ into squares with
diameter $\leq 2 \eta_1 / (\kappa \delta_\ell)$, considering these squares'
midpoints $\tilde c$, and projecting them by $c:=\tilde c / \|\tilde c\|$
to the unit sphere.
The two-dimensional case is illustrated in
Figure~\ref{fi:projected_directions}.
By construction, each point on the cube's surface has a distance of
less than $\eta_1 / (\kappa \delta_\ell)$ to one of the midpoints, and we
only have to prove that the same holds for the points' projections to
the unit sphere.

%
% Lemma: Projection
%
\begin{lemma}[Projection]
\label{le:projection}
Let $x,y\in\bbbr^n$ with $\|x\|,\|y\|\geq 1$.
We have
\begin{equation*}
  \left\| \frac{x}{\|x\|} - \frac{y}{\|y\|} \right\|
  \leq \|x-y\|.
\end{equation*}
\end{lemma}
\begin{proof}
We assume $1\leq \|x\|\leq\|y\|$ without loss of generality
and define
\begin{equation*}
  \tilde y := \frac{\langle x, y \rangle}{\|x\|\,\|y\|^2} y
  \qquad\text{ such that }\qquad
  \langle \frac{x}{\|x\|} - \tilde y, y \rangle
  = \frac{\langle x, y \rangle}{\|x\|}
    - \frac{\langle x, y \rangle}{\|x\|} = 0,
\end{equation*}
we can apply Pythagoras' equation and the Cauchy-Schwarz
inequality to find
\begin{align*}
  \left\| \frac{x}{\|x\|} - \frac{y}{\|y\|} \right\|^2
  &= \left\| \frac{x}{\|x\|} - \tilde y + 
             \left( \frac{\langle x, y \rangle}{\|x\|\,\|y\|^2}
                    - \frac{1}{\|y\|} \right) y \right\|^2\\
  &= \left\| \frac{x}{\|x\|} - \tilde y \right\|^2
     + \left( \frac{\|x\|\,\|y\| - \langle x, y \rangle}
                   {\|x\|\,\|y\|^2} \right)^2 \|y\|^2\\
  &\leq \left\| \frac{x}{\|x\|} - \tilde y \right\|^2
     + \left( \frac{\|y\|^2 - \langle x, y \rangle}
                   {\|x\|\,\|y\|^2} \right)^2 \|y\|^2\\
  &= \left\| \frac{x}{\|x\|} - \tilde y
             + \left(\frac{\langle x, y \rangle}{\|x\|\,\|y\|^2}
                   - \frac{1}{\|x\|} \right) y \right\|^2
   = \frac{\|x-y\|^2}{\|x\|^2} \leq \|x-y\|^2.
\end{align*}
\end{proof}

The construction of a suitable block tree can now be accomplished by
standard algorithms \cite{GRHA02,BO10} if we replace the standard admissibility
condition by (\ref{eq:adm_parabolic}) and (\ref{eq:adm_standard}).

% ============================================================
%
% Matrix-vector multiplication
%
% ============================================================
\section{Matrix-vector multiplication}

Let $G$ be a $\mathcal{DH}^2$-matrix for the directional cluster
bases $V$ and $W$, and let $x\in\bbbc^\Idx$.
We are looking for an efficient algorithm for evaluating the
matrix-vector product $y=Gx$.

We can follow the familiar approach of fast multipole and
$\mathcal{H}^2$-matrix techniques:
considering that the submatrices are factorized into three
terms
\begin{align*}
  G|_{\hat t\times\hat s} &= V_{tc} S_b W_{sc}^* &
  &\text{ for all } b=(t,s)\in\lfaII,
\end{align*}
the algorithm is split into three phases:
in the first phase, called the \emph{forward transformation},
we multiply by $W_{sc}^*$ and compute
\begin{subequations}
\begin{align}
  \widehat{x}_{sc} &= W_{sc}^* x|_{\hat s} &
  &\text{ for all } s\in\ctI,\ c\in\mathcal{D}_s,\label{eq:mvm1}
\intertext{in the second phase, the \emph{coupling step},
we multiply these coefficient vectors by the coupling matrices $S_b$
and obtain}
  \widehat{y}_{tc}
  &:= \sum_{\substack{b=(t,s)\in\ctII\\ c=c_b}} S_b \widehat{x}_{sc} &
  &\text{ for all } t\in\ctI,\ c\in\mathcal{D}_t,\label{eq:mvm2}
\intertext{and in the final phase, the \emph{backward transformation},
we multiply by $V_{tc}$ to get the result}
  y_i &= \sum_{\substack{t\in\ctI,\ c\in\mathcal{D}_t\\ i\in\hat t}}
           (V_{tc} \widehat{y}_{tc})_i &
  &\text{ for all } i\in\Idx.\label{eq:mvm3}
\end{align}
\end{subequations}
The first and third phase can be handled efficiently by using
the transfer matrices $E_{t'c}$:
let $s\in\ctI$ with $\sons(s)\neq\emptyset$, and let $c\in\mathcal{D}_s$
and $c':=\sd{s}(c)$.
Due to Definition~\ref{de:cluster}, the set $\{\hat s'\ :\ s'\in\sons(s)\}$
is a disjoint partition of the index set $\hat s$.
Combined with (\ref{eq:nested}), this implies
\begin{equation*}
  W_{sc}^* x|_{\hat s}
   = \sum_{s'\in\sons(s)} (W_{sc}|_{\hat s'\times k})^* x|_{\hat s'}
   = \sum_{s'\in\sons(s)} E_{s'c}^* W_{s'c'}^* x|_{\hat s'}
   = \sum_{s'\in\sons(s)} E_{s'c}^* \widehat{x}_{s'c'},
\end{equation*}
and we can prepare \emph{all} coefficient vectors $\widehat{x}_{sc}$
by the simple recursion given on the left of Figure~\ref{fi:forward_backward}.
By similar arguments we find that the third phase can also be handled
by the recursion given on the right of Figure~\ref{fi:forward_backward}.

%
% Figure: Forward and backward transformations
%
\begin{figure}
  \begin{minipage}[t]{0.47\textwidth}
  \begin{tabbing}
    \textbf{procedure} forward($s$, $x$, \textbf{var} $\widehat{x}$);\\
    \textbf{if} $\sons(s)=\emptyset$ \textbf{then}\\
    \quad\= \textbf{for} $c\in\mathcal{D}_s$ \textbf{do}
               $\widehat{x}_{sc} \gets W_{sc}^* x|_{\hat s}$\\
    \textbf{else begin}\\
    \> \textbf{for} $s'\in\sons(s)$ \textbf{do}
               forward($s'$, $x$, $\widehat{x}$);\\
    \> \textbf{for} $c\in\mathcal{D}_s$ \textbf{do begin}\\
    \> \quad\= $\widehat{x}_{sc} \gets 0$;\quad
               $c' \gets \sd{s}(c)$;\\
    \> \> \textbf{for} $s'\in\sons(s)$ \textbf{do}\\
    \> \> \quad\= $\widehat{x}_{sc} \gets \widehat{x}_{sc}
                  + E_{s'c}^* \widehat{x}_{s'c'}$\\
    \> \textbf{end}\\
    \textbf{end}
  \end{tabbing}
  \end{minipage}%
  \hfill%
  \begin{minipage}[t]{0.47\textwidth}
  \begin{tabbing}
    \textbf{procedure} backward($t$, \textbf{var} $\widehat{y}$, $y$);\\
    \textbf{if} $\sons(t)=\emptyset$ \textbf{then}\\
    \quad\= \textbf{for} $c\in\mathcal{D}_s$ \textbf{do}
               $y|_{\hat t} \gets y|_{\hat t} + V_{tc} \widehat{y}_{tc}$\\
    \textbf{else begin}\\
    \> \textbf{for} $c\in\mathcal{D}_t$ \textbf{do}\\
    \> \quad\= \textbf{for} $t'\in\sons(t)$ \textbf{do begin}\\
    \> \> \quad\= $c' \gets \sd{t}(c)$;\\
    \> \> \> $\widehat{y}_{t'c'} \gets \widehat{y}_{t'c'}
                  + E_{t'c} \widehat{y}_{tc}$\\
    \> \> \textbf{end};\\
    \> \textbf{for} $t'\in\sons(t)$ \textbf{do}
               backward($t'$, $\widehat{y}$, $y$)\\
    \textbf{end}
  \end{tabbing}
  \end{minipage}

  \caption{Fast forward and backward transformation}
  \label{fi:forward_backward}
\end{figure}

The submatrices corresponding to inadmissible leaves $b=(t,s)\in\lfII$
are stored as standard arrays and can be evaluated accordingly.

We can see that the algorithm performs exactly one
matrix-vector multiplication with each of the matrices appearing in the
$\mathcal{DH}^2$-matrix representation:
the forward transformation uses $W_{sc}$ in the leaves and
$E_{s'c}$ in the other clusters.
The coupling step uses $S_b$ for the admissible leaves
$b=(t,s)\in\lfII$.
The backward transformation uses $V_{tc}$ in the leaves and
$E_{t'c}$ in the other clusters.
The nearfield computation uses $G|_{\hat t\times\hat s}$ for the
inadmissible leaves $b=(t,s)\in\lfII$.
This means that finding a bound for the storage requirements of
the $\mathcal{DH}^2$-matrix representation also provides us with
a bound for the computational work for the matrix-vector
multiplication.

% ============================================================
%
% Complexity
%
% ============================================================
\section{Complexity}

In order to obtain an upper bound for the storage requirements,
we follow the approach presented in \cite{BEKUVE15}:
as in the standard theory (cf. \cite{GRHA02}), we start by investigating
the cardinalities of the sets
\begin{align*}
  \brow(t) &:= \{ s\in\ctI\ :\ (t,s)\in\ctII \} &
  &\text{ for all } t\in\ctI,\\
  \bcol(s) &:= \{ t\in\ctI\ :\ (t,s)\in\ctII \} &
  &\text{ for all } s\in\ctI.
\end{align*}
Since the admissibility condition is symmetric, we have
$\#\brow(t)=\#\bcol(t)$ and can therefore focus on deriving
bounds for $\#\brow(t)$.

For the sake of simplicity, we assume that the bounding boxes for
clusters $t\in\ctIl{\ell}$ on a given level $\ell$ are
translation-equivalent, i.e., that there is a ``reference box''
$B_\ell$ such that for all clusters $t\in\ctIl{\ell}$ we can find
$d\in\bbbr^3$ satisfying
\begin{equation}\label{eq:bbox_similar}
  B_t = B_\ell + d.
\end{equation}
This property can be easily guaranteed during the construction of
the bounding boxes by adding suitable padding.

Since the bounding box $B_t$ of a cluster $t\in\ctI$ should not
be significantly larger than the union of the bounding boxes of
its sons, we can assume that there is a constant $\Csb\in\bbbr_{\geq 1}$
such that
\begin{align}\label{eq:son_boxes}
  \diam(B_t) &\leq \Csb \diam(B_{t'}) &
  &\text{ for all } t\in\ctI,\ t'\in\sons(t).
\end{align}
We use the minimal block tree satisfying the admissibility condition,
therefore a block has sons only if it is inadmissible.
Assuming that there is a constant $\Csn\geq 1$ such that
\begin{align}\label{eq:sons_bound}
  \#\sons(t) &\leq \Csn, &
  \#\sons(t) &\neq 1 &
  &\text{ for all } t\in\ctI
\end{align}
holds, bounding the number of inadmissible blocks therefore
gives rise to a bound for all blocks.

To find an estimate of this kind, we make use of the surface
measure in $\Omega$:
we will prove that inadmissible blocks have to correspond to
subsets of $\Omega$ that are geometrically close to each
other and then show that there can be only a limited number
of these blocks.
To this end, we consider the intersection of three-dimensional
balls with the surface $\Omega$.
Let
\begin{align*}
  \mathcal{B}(x,r)
  &:= \{ y\in\bbbr^3\ :\ \|y-x\| \leq r \} &
  &\text{ for all } x\in\bbbr^3,\ r\in\bbbr_{\geq 0}
\end{align*}
denote the three-dimensional ball of radius $r$ around $x$.
We require the surface $\Omega$ to be ``reasonably similar'' to a
two-dimensional plane, i.e., that there are constants
$\Cbp,\Cbb\in\bbbr_{>0}$ such that
\begin{subequations}\label{eq:planar}
\begin{align}
  |\Omega\cap\mathcal{B}(x,r)|
  &\leq \Cbp r^2 &
  &\text{ for all } x\in\bbbr^3,\ r\in\bbbr_{\geq 0},
      \label{eq:planar_ball}\\
  \diam^2(B_t) &\leq \Cbb |B_t\cap\Omega| &
  &\text{ for all } t\in\ctI,
      \label{eq:planar_bbox}
\end{align}
\end{subequations}
where $|X|$ denotes the surface measure of a measurable
set $X\subseteq\Omega$.

In order to be able to draw conclusions about the clusters $t\in\ctI$
based on the sets $B_t\cap\Omega$, we have to limit the overlap of
these sets, i.e., we require that there is a constant
$\Cov\in\bbbr_{>0}$ such that
\begin{align}\label{eq:overlap}
  \#\{ t\in\ctIl{\ell}\ :\ x\in B_t \}
  &\leq \Cov &
  &\text{ for all } x\in\Omega,\ \ell\in\bbbn_0.
\end{align}

In order to estimate the number of clusters, we assume that there
is a constant $\Crs\in\bbbr_{>0}$ such that
\begin{subequations}\label{eq:leaves_bound}
\begin{align}
  \Csb \kappa \diam(B_t) &\leq 1 &
  &\text{ for all leaves } t\in\lfI,\label{eq:lowfreq_leaves}\\
  \Crs^{-1} k &\leq \#\hat t \leq \Crs k &
  &\text{ for all leaves } t\in\lfI.\label{eq:rank_leaves}.
\end{align}
\end{subequations}
Finally we require a weak mesh regularity assumption:
clusters close to a leaf cluster should not be significantly larger
than the leaf, i.e., we assume that there is a constant
$\Cun\in\bbbr_{>0}$ such that
\begin{align}\label{eq:cluster_uniform}
  \eta_2 \dist(B_t, B_s) < \diam(B_t)
  &\Rightarrow \#\hat s \leq \Cun \#\hat t &
  &\text{ for all } t\in\lfI,\ s\in\ctI,\ \level(t)=\level(s).
\end{align}
Using the assumptions (\ref{eq:bbox_similar}) to
(\ref{eq:cluster_uniform}), we can now proceed to prove the required
complexity estimates.

%
% Lemma: Sparsity
%
\begin{lemma}[Sparsity]
\label{le:sparsity}
Let (\ref{eq:bbox_similar}), (\ref{eq:son_boxes}),
(\ref{eq:sons_bound}), (\ref{eq:planar}), and (\ref{eq:overlap}) hold.
We have
\begin{align}\label{eq:sparsity}
  \#\brow(t),\#\bcol(t)
  &\leq \begin{cases}
          \Csp &\text{ if } \Csb \kappa \diam(B_t) \leq 1,\\
          \Csp \kappa^2 \diam^2(B_t) &\text{ otherwise}
        \end{cases} &
  &\text{ for all } t\in\ctI,
\end{align}
with $\Csp := \Csn \Csb^2 \Cbb
\Cov \Cbp (3/2 + 1/\eta_2)^2$.
\end{lemma}
\begin{proof}
Let $t\in\ctI$.

We denote the set of inadmissible blocks connected to a
cluster $t\in\ctI$ by
\begin{equation*}
  C_t := \{ s\in\ctI\ :\ (t,s)\in\ctII \text{ is inadmissible} \}.
\end{equation*}
Given a non-root block $(t,s)\in\ctII$, its father $(t^+,s^+)$
satisfies $s^+\in C_{t^+}$ by construction, so we can use
(\ref{eq:sons_bound}) to find
\begin{align}
  \brow(t) &\subseteq \bigcup_{s^+\in C_{t^+}} \sons(s^+), &
  \#\brow(t) &\leq \sum_{s^+\in C_{t^+}} \#\sons(s^+)
              \leq \Csn \#C_{t^+}.\label{eq:brow_Ct}
\end{align}
Our goal is now to bound the cardinality of the sets $C_t$.

We first consider the case $\kappa \diam(B_t) \leq 1$.
In this case, (\ref{eq:adm_standard}) implies (\ref{eq:adm_parabolic}),
so for each $s\in C_t$ the condition (\ref{eq:adm_standard}) does
not hold.
This implies $\eta_2 \dist(B_t,B_s) < \diam(B_t)$, i.e., there are
$x\in B_t$ and $y\in B_s$ such that
\begin{equation*}
  \|x-y\| < \diam(B_t) / \eta_2.
\end{equation*}
Let $m_t\in B_t$ again denote the midpoint of $B_t$ and let $z\in B_s$.
By (\ref{eq:bbox_similar}) and the triangle inequality, we have
\begin{align*}
  \|z-m_t\|
  &\leq \|z-y\| + \|y-x\| + \|x-m_t\|\\
  &< \diam(B_s) + \diam(B_t) / \eta_2 + \diam(B_t) / 2
   = (3/2 + 1/\eta_2) \diam(B_t),
\end{align*}
i.e., $B_s \subseteq \mathcal{B}_t
:= \mathcal{B}(m_t, (3/2 + 1/\eta_2) \diam(B_t))$.
Using (\ref{eq:bbox_similar}), (\ref{eq:planar}) and (\ref{eq:overlap}),
we obtain
\begin{align*}
  \diam^2(B_t) \#C_t
  &= \sum_{s\in C_t} \diam^2(B_s)
   \leq \sum_{s\in C_t} \Cbb |B_s\cap\Omega|
   = \Cbb \int_{\mathcal{B}_t\cap\Omega} \sum_{s\in C_t} 1_{B_s}(x) \,dx\\
  &\leq \Cbb \int_{\mathcal{B}_t\cap\Omega} \Cov \,dx
   = \Cbb \Cov |\mathcal{B}_t\cap\Omega|\\
  &\leq \Cbb \Cov \Cbp (3/2 + 1/\eta_2)^2 \diam^2(B_t)
   = \widehat{C}_\text{sp} \diam^2(B_t)
\end{align*}
with $\widehat{C}_\text{sp} := \Cbb \Cov \Cbp
(3/2 + 1/\eta_2)^2$, and conclude $\#C_t\leq \widehat{C}_\text{sp}$.

Let us consider the case $\kappa \diam(B_t) > 1$.
Now (\ref{eq:adm_parabolic}) implies (\ref{eq:adm_standard}), so for
each $s\in C_t$ the condition (\ref{eq:adm_parabolic}) does not hold.
This implies $\eta_2 \dist(B_t, B_s) < \kappa \diam^2(B_t)$.
By the same arguments as before we find
\begin{align*}
  \|z-m_t\|
  &< \diam(B_s) + \kappa \diam^2(B_t) / \eta_2 + \diam(B_t) / 2\\
  &= \kappa \diam^2(B_t) / \eta_2 + 3 \diam(B_t) / 2
   < \kappa \diam^2(B_t) / \eta_2 + 3 \kappa \diam^2(B_t) / 2\\
  &= \kappa (3/2 + 1/\eta_2) \diam^2(B_t)
     \qquad\text{ for all } z\in B_s
\end{align*}
and conclude
\begin{equation*}
  \diam^2(B_t) \#C_t
  \leq \Cbb \Cov \Cbp \kappa^2 (3/2 + 1/\eta_2)
           \diam^4(B_t)
  = \widehat{C}_\text{sp} \kappa^2 \diam^4(B_t),
\end{equation*}
which gives us
\begin{equation*}
  \#C_t \leq \widehat{C}_\text{sp} \kappa^2 \diam^2(B_t).
\end{equation*}
Now we can return to the final result.
If $t$ is the root, we have $\#\brow(t)=\#\bcol(t)=1$ and the
estimate is trivial.

Let $t\neq\treeroot(\ctI)$, and let $t^+\in\ctI$ denote its father.
If $\Csb \kappa \diam(B_t) \leq 1$ holds, (\ref{eq:son_boxes})
yields $\kappa \diam(B_{t^+}) \leq \Csb \kappa \diam(B_t) \leq 1$
and we obtain $\#C_{t^+}\leq\widehat{C}_\text{sp}$ and therefore
$\#\brow(t)\leq \Csn \widehat{C}_\text{sp}\leq \Csp$ using
(\ref{eq:brow_Ct}).

Otherwise, i.e., if $\kappa \diam(B_{t^+}) > 1$ holds, we have
$C_{t^+}\leq \widehat{C}_{\text{sp}} \kappa^2 \diam^2(B_{t^+})$, and we can use
(\ref{eq:son_boxes}) to get
$C_{t^+}\leq \widehat{C}_\text{sp} \Csb^2 \kappa^2 \diam^2(B_t)$.
Applying (\ref{eq:brow_Ct}) completes the proof.
\end{proof}

The estimate (\ref{eq:sparsity}) is our generalization of the
\emph{sparsity assumption} used in standard $\mathcal{H}$-matrix
methods \cite{GRHA02}.

%
% Lemma: Clusters
%
\begin{lemma}[Clusters]
\label{le:clusters}
Let (\ref{eq:bbox_similar}), (\ref{eq:sons_bound}),
(\ref{eq:planar_bbox}), (\ref{eq:overlap}) and (\ref{eq:rank_leaves}) hold.
We have
\begin{subequations}\label{eq:tree_bound}
\begin{align}
  \#\ctIl{\ell}
  &\leq \Clv \frac{|\Omega|}{\diam^2(B_\ell)} &
  &\text{ for all } \ell\in\bbbn_0,\label{eq:level_bound}\\
  \#\ctI
  &\leq \Clv \#\Idx / k\label{eq:cluster_bound}
\end{align}
\end{subequations}
with $\Clv := \max\{ \Cbb \Cov, 2 \Crs \}$.
\end{lemma}
\begin{proof}
Let $\ell\in\bbbn_0$.
Combining (\ref{eq:bbox_similar}), (\ref{eq:planar_bbox}) and
(\ref{eq:overlap}) yields
\begin{align*}
  \diam^2(B_\ell) \#\ctIl{\ell}
  &= \sum_{t\in\ctIl{\ell}} \diam^2(B_t)
   \leq \Cbb \sum_{t\in\ctIl{\ell}} |B_t\cap\Omega|
   = \Cbb \int_\Omega \sum_{t\in\ctIl{\ell}} 1_{B_t}(x) \,dx\\
  &\leq \Cbb \Cov \int_\Omega 1 \,dx
   = \Cbb \Cov |\Omega|
   \leq \Clv |\Omega|,
\end{align*}
and dividing by $\diam^2(B_\ell)$ gives us the bound for $\#\ctIl{\ell}$.

For the bound for $\#\ctI$, we notice that Definition~\ref{de:cluster}
implies that the sets $\hat t$ for leaves $t\in\lfI$ are pairwise
disjoint.
Using (\ref{eq:rank_leaves}), we find
\begin{equation*}
  \#\lfI = \sum_{t\in\lfI} 1
  \leq \sum_{t\in\lfI} \Crs \frac{\#\hat t}{k}
  = \frac{\Crs}{k} \sum_{t\in\lfI} \#\hat t
  = \frac{\Crs}{k} \#\bigcup_{t\in\lfI} \hat t
  = \frac{\Crs}{k} \#\Idx.
\end{equation*}
With (\ref{eq:sons_bound}), we have $\sons(t)\neq 1$ for all clusters,
and a simple induction yields
\begin{equation*}
  \#\ctI \leq 2 \#\lfI - 1
  < 2 \Crs \#\Idx / k
  \leq \Clv \#\Idx / k.
\end{equation*}
\end{proof}

%
% Lemma: Directional cluster basis
%
\begin{lemma}[Directional cluster basis]
\label{le:storage_basis}
Let (\ref{eq:bbox_similar}), (\ref{eq:sons_bound}), (\ref{eq:planar}),
(\ref{eq:overlap}) and (\ref{eq:leaves_bound}) hold, and
let $p_\Idx>0$.
A directional cluster basis of rank $k$ requires not more than
\begin{equation*}
  \Ccb ( k \#\Idx + p_\Idx k^2 \kappa^2 ) \text{ units of storage,}
\end{equation*}
where
$\Ccb := 2 \Csp (1 + \Csn \Clv \max\{ 1, |\Omega| \})$.
\end{lemma}
\begin{proof}
Let $t\in\ctI$.
Since we only have to store matrices for directions $c\in\mathcal{D}_t$
that are actually used, we have no more than $\#\brow(t)+\#\bcol(t)$
directions to consider, and Lemma~\ref{le:sparsity} gives us
\begin{equation}\label{eq:direction_bound}
  \#\mathcal{D}_t
  \leq \begin{cases}
         2 \Csp &\text{ if } \Csb \kappa \diam(B_t)\leq 1,\\
         2 \Csp \kappa^2 \diam^2(B_t) &\text{ otherwise}.
       \end{cases}
\end{equation}
First consider the case that $t$ is a leaf.
We store $V_{tc}\in\bbbc^{\hat t\times k}$ for each direction $c\in\mathcal{D}_t$.
Due to (\ref{eq:lowfreq_leaves}), we have $\Csb\kappa \diam(B_t) \leq 1$
and (\ref{eq:direction_bound}) yields $\#\mathcal{D}_t \leq 2 \Csp$.
Definition~\ref{de:cluster} implies that the index sets $\hat t$
of leaf clusters are pairwise disjoint, and we obtain the bound
\begin{align*}
  \sum_{t\in\lfI} \sum_{c\in\mathcal{D}_t} (\#\hat t) k
  &\leq 2 \Csp \sum_{t\in\lfI} (\#\hat t) k
   = 2 \Csp k \sum_{t\in\lfI} \#\hat t
   = 2 \Csp k \#\bigcup_{t\in\lfI} \hat t
   = 2 \Csp k \#\Idx
\end{align*}
for the storage requirements of all leaf matrices.

If $t$ is not a leaf, we store $E_{t'c}\in\bbbc^{k\times k}$ for
each son $t'\in\sons(t)$ and each direction $c\in\mathcal{D}_t$.
Due to Lemma~\ref{le:clusters} and (\ref{eq:direction_bound}), this
requires not more than
\begin{align*}
  \sum_{\substack{t\in\ctI,\\\sons(t)\neq\emptyset}}
      \sum_{t'\in\sons(t)} \sum_{c\in\mathcal{D}_t} k^2
  &\leq \Csn k^2 \sum_{t\in\ctI\setminus\lfI} \#\mathcal{D}_t\\
  &= \Csn k^2 
     \sum_{\substack{t\in\ctI\setminus\lfI\\ \Csb \kappa \diam(B_t)\leq 1}} \#\mathcal{D}_t
   + \Csn k^2
     \sum_{\substack{t\in\ctI\setminus\lfI\\ \Csb \kappa \diam(B_t)>1}} \#\mathcal{D}_t\\
  &\leq \Csn k^2 2 \Csp \#\ctI
     + \Csn k^2
     \sum_{\ell=0}^{p_\Idx-1}
     \sum_{t\in\ctIl{\ell}} 2 \Csp \kappa^2 \diam^2(B_\ell)\\
  &\leq 2 \Csp \Csn \Clv k \#\Idx
     + 2 \Csp \Csn k^2 p_\Idx \kappa^2 \Clv |\Omega|\\
  &\leq 2 \Csp \Csn \Clv \max\{1,|\Omega|\}
        (k \#\Idx + p_\Idx k^2 \kappa^2).
\end{align*}
Combining the estimates for leaf and non-leaf cluster gives us the
desired estimate.
\end{proof}

%
% Lemma: Nearfield and coupling matrices
%
\begin{lemma}[Nearfield and coupling matrices]
\label{le:storage_blocks}
Let (\ref{eq:bbox_similar}), (\ref{eq:sons_bound}), (\ref{eq:planar}),
(\ref{eq:overlap}), (\ref{eq:leaves_bound}) and (\ref{eq:cluster_uniform})
hold.
Nearfield and coupling matrices require not more than
\begin{equation*}
  \Cnc (k \#\Idx + (p_{\Idx} + 1) k^2 \kappa^2)
  \text{ units of storage,}
\end{equation*}
where $\Cnc := \Clb \Csp \Clv \max\{ 1, |\Omega| \}$
and $\Clb := \max\{ 1, \Crs^2 \Cun \}$.
\end{lemma}
\begin{proof}
Let $b=(t,s)\in\lfII$.

If $b$ is admissible, we store the $k\times k$ coupling matrix $S_b$
in $k^2$ units of storage.

If $b$ is not admissible, we store the matrix $G|_{\hat t\times\hat s}$.
Due to our construction, an inadmissible block can only appear if
$t$ or $s$ is a leaf.
Without loss of generality we assume that $t$ is a leaf, and
(\ref{eq:rank_leaves}) yields $\#\hat t\leq \Crs k$.
Our construction also guarantees $\level(t)=\level(s)$, so
(\ref{eq:cluster_uniform}) gives us
\begin{equation*}
   \#\hat s\leq \Cun \#\hat t \leq \Cun \Crs k,
\end{equation*}
and we conclude that the nearfield matrix requires not more than
$\Crs^2 \Cun k^2$ units of storage.

Since each leaf block takes not more than $\Clb k^2$ units
of storage, the total storage requirements are bounded by
\begin{equation}\label{eq:storage_matrices_1}
  \sum_{b=(t,s)\in\lfII} \Clb k^2
  \leq \sum_{t\in\ctI} \sum_{s\in\brow(t)} \Clb k^2
  = \Clb k^2 \sum_{t\in\ctI} \#\brow(t).
\end{equation}
Let $t\in\ctI$.
If $\Csb \kappa \diam(B_t) \leq 1$, we have
$\#\brow(t)\leq \Csp$ due to (\ref{eq:sparsity}).
Otherwise we have $\#\brow(t)\leq \Csp \kappa^2 \diam^2(B_t)$.
Combining both estimates with (\ref{eq:level_bound}) and
(\ref{eq:cluster_bound}) yields
\begin{align*}
  \sum_{t\in\ctI} \#\brow(t)
  &= \sum_{\substack{t\in\ctI\\ \Csb \kappa \diam(B_t)\leq 1}} \#\brow(t)
     + \sum_{\substack{t\in\ctI\\ \Csb \kappa \diam(B_t)>1}} \#\brow(t)\\
  &\leq \sum_{t\in\ctI} \Csp
     + \sum_{\ell=0}^{p_\Idx} \sum_{t\in\ctIl{\ell}}
          \Csp \kappa^2 \diam^2(B_t)
   \leq \Csp \#\ctI
     + \Csp \Clv \kappa^2 \sum_{\ell=0}^{p_\Idx} |\Omega|\\
  &\leq \frac{\Csp \Clv}{k} \#\Idx
     + \Csp \Clv \kappa^2 (p_\Idx+1) |\Omega|,
\end{align*}
and with (\ref{eq:storage_matrices_1}) we obtain the desired result.
\end{proof}

%
% Theorem: Complexity
%
\begin{theorem}[Complexity]
\label{th:complexity}
Let (\ref{eq:bbox_similar}), (\ref{eq:sons_bound}), (\ref{eq:planar}),
(\ref{eq:overlap}), (\ref{eq:leaves_bound}) and (\ref{eq:cluster_uniform})
hold.

A $\mathcal{DH}^2$-matrix representation of a matrix
$G\in\bbbc^{\Idx\times\Idx}$ requires not more than
\begin{align*}
  \Cdh (k \#\Idx + (p_\Idx+1) k^2 \kappa^2)
    &\text{ units of storage}
\intertext{and a matrix-vector multiplication requires not more than}
  2 \Cdh (k \#\Idx + (p_\Idx+1) k^2 \kappa^2)
    &\text{ operations,}
\end{align*}
where $\Cdh := \Cnc + \Ccb$.
\end{theorem}
\begin{proof}
Combine Lemma~\ref{le:storage_basis} and Lemma~\ref{le:storage_blocks}.
\end{proof}

%
% Remark: Asymptotic complexity
%
\begin{remark}[Asymptotic complexity]
Let $n:=\#\Idx$ denote the matrix dimension.
In order to resolve waves of wavelength $\sim 1/\kappa$, we typically
have to choose $n\sim\kappa^2$.

Standard cluster algorithms ensure $p_\Idx\sim\log(n)$ for regular
meshes, so Theorem~\ref{th:complexity} states that storage requirements
and computational complexity are in $\mathcal{O}(n k^2 \log(n))$.
\end{remark}

As mentioned before, using a representation that reduces
to standard $\mathcal{H}^2$-matrices in the low-frequency regime offers
that advantage of obtaining linear complexity, while the approach
presented in \cite{BEKUVE15} switches to $\mathcal{H}$-matrices and
therefore can reach only linear-logarithmic complexity.

Our complexity result is comparable to the ones obtained in
\cite{MESCDA12}, but we can expect significantly lower ranks $k$,
since we are free to apply SVD-based quasi-optimal compression to
all matrices, not only to the coupling (or M2L) matrices $S_b$.

% ============================================================
%
% Compression algorithm
%
% ============================================================
\section{Compression algorithm}

We have introduced a matrix representation that matches the
approximation scheme described in \cite{MESCDA12}, and we have
proven that it can be efficient if the rank $k$ is small.
Our goal is now to develop an algorithm that can approximate
an \emph{arbitrary} matrix by an $\mathcal{DH}^2$-matrix, since
it could lead the way to efficient recompression schemes or
preconditioners.

Given a matrix $G\in\bbbc^{\Idx\times\Idx}$, a cluster tree $\ctI$,
a family $(\mathcal{D}_t)_{t\in\ctI}$ of hierarchical directions,
and a block tree $\ctII$ with admissible leaves $\lfaII$ and
inadmissible leaves $\lfiII$, we are looking for an algorithm
that constructs a $\mathcal{DH}^2$-matrix approximation of $G$
with a prescribed accuracy.
Our approach is to extend the algorithm introduced in \cite{BOHA02}
to fit the more general structure of $\mathcal{DH}^2$-matrices.

For the sake of numerical stability and efficiency, we focus
on orthogonal directional cluster bases.

%
% Definition: Orthogonality
%
\begin{definition}[Orthogonality]
We call a directional cluster basis $Q=(Q_{tc})_{t\in\ctI,c\in\mathcal{D}_t}$
\emph{orthogonal} if
\begin{align*}
  Q_{tc}^* Q_{tc} &= I &
  &\text{ for all } t\in\ctI,\ c\in\mathcal{D}_t,
\end{align*}
i.e., if the columns of each matrix $Q_{tc}$ are an orthonormal
basis of its range.
\end{definition}

If $V$ and $W$ are orthogonal directional cluster bases,
$V_{tc} V_{tc}^*$ and $W_{sc} W_{sc}^*$ are orthogonal projections,
and for $b=(t,s)\in\lfaII$, $c=c_b$ the matrix
\begin{align*}
  V_{tc} V_{tc}^* G|_{\hat t\times\hat s} W_{sc} W_{sc}^*
  &= V_{tc} \widehat{S}_b W_{sc}^*, &
  \widehat{S}_b &:= V_{tc}^* G|_{\hat t\times\hat s} W_{sc}
\end{align*}
is the best approximation of $G|_{\hat t\times\hat s}$ of the
shape (\ref{eq:matrix_apx}) with respect to the Frobenius norm
and close to the best approximation with respect to the spectral norm.

For the spectral norm, we find
\begin{align*}
  \|G|_{\hat t\times\hat s}
    &- V_{tc} V_{tc}^* G|_{\hat t\times\hat s} W_{sc} W_{sc}^*\|^2\\
  &\leq \|G|_{\hat t\times\hat s}
          - V_{tc} V_{tc}^* G|_{\hat t\times\hat s}\|^2
      + \|V_{tc} V_{tc}^* (G|_{\hat t\times\hat s}
            - G|_{\hat t\times\hat s} W_{sc} W_{sc}^*) \|^2\\
  &\leq \|G|_{\hat t\times\hat s}
          - V_{tc} V_{tc}^* G|_{\hat t\times\hat s}\|^2
      + \|G|_{\hat t\times\hat s}^*
          - W_{sc} W_{sc}^* G|_{\hat t\times\hat s}^*\|^2,
\end{align*}
so we can focus on the construction of the row basis, since
the column basis can be obtained by applying our procedure
to the adjoint matrix $G^*$.

We denote the new directional cluster basis by
$Q=(Q_{tc})_{t\in\ctI,c\in\mathcal{D}_t}$.
Due to the nested structure (\ref{eq:nested}), a cluster basis
matrix $Q_{tc}$ not only has to be able to approximate
submatrices $G|_{\hat t\times\hat s}$ for $b=(t,s)\in\lfaII$,
but also submatrices corresponding to ancestors of $t$.
We define the sets of \emph{descendants} of clusters and directions
inductively as
\begin{align*}
  \desc(t) &:= \begin{cases}
    \{t\} &\text{ if } \sons(t)=\emptyset,\\
    \{t\} \cup \bigcup_{t'\in\sons(t)} \desc(t') &\text{ otherwise}
  \end{cases} &
  &\text{ for all } t\in\ctI,\\
  \descd{t}(c) &:= \begin{cases}
    \{c\} &\text{ if } \sons(t)=\emptyset,\\
    \{c\} \cup \descd{t'}(\sd{t}(c)) &\text{ otherwise, with } t'\in\sons(t)
  \end{cases} &
  &\text{ for all } t\in\ctI,\ c\in\mathcal{D}_t,
\end{align*}
and collect the submatrices that have to be approximated by $Q_{tc}$ in
the matrices
\begin{align*}
  G_{tc} &:= G|_{\hat t\times \mathcal{F}_{tc}},\\
  F_{tc} &:= \{ s\in\ctI\ :\ \exists t^+\in\ctI\ :\ 
              t\in\desc(t^+),\ b=(t^+,s)\in\lfaII,\ c\in\descd{t^+}(c_b) \},\\
  \mathcal{F}_{tc} &:= \bigcup \{ \hat s\ :\ s\in F_{tc} \}
       \qquad\text{ for all } t\in\ctI,\ c\in\mathcal{D}_t.
\end{align*}
For each $t\in\ctI$ and $c\in\mathcal{D}_t$, we have to find
a matrix $Q_{tc}$ of low rank such that
\begin{equation}\label{eq:approx_Gtc}
  \|G_{tc} - Q_{tc} Q_{tc}^* G_{tc}\|
  \leq \epsilon_t
\end{equation}
holds for a suitable $\epsilon_t>0$.

If $t$ is a leaf of $\ctI$, we can solve this problem by computing
the singular value decomposition of $G_{tc}$ and using the first
$k$ left singular vectors as the columns of $Q_{tc}$.
If we assume that leaf clusters correspond to only small sets of
indices, this procedure is quite efficient.

If $t$ is not a leaf, we have to take (\ref{eq:nested}) into account.
For the sake of simplicity we will only consider the case
$\sons(t)=\{t_1,t_2\}$ with $t_1\neq t_2$.
Definition~\ref{de:directions} gives us a ``son direction''
$c':=\sd{t}(c)\in\mathcal{D}_{t_1}=\mathcal{D}_{t_2}$, and we have to find
transfer matrices $E_{t_1c}$ and $E_{t_2c}$ such that
\begin{align*}
  \|G_{tc} - Q_{tc} Q_{tc}^* G_{tc}\|
  &\leq \epsilon, &
  Q_{tc}
  &= \begin{pmatrix}
       Q_{t_1c'} E_{t_1c}\\
       Q_{t_2c'} E_{t_2c}
     \end{pmatrix}.
\end{align*}
Substituting $Q_{tc}$ in the left inequality and using
Pythagoras' equation yields
\begin{align*}
  \|G_{tc} &- Q_{tc} Q_{tc}^* G_{tc}\|^2
   = \left\| \begin{pmatrix}
               G|_{\hat t_1\times\mathcal{F}_{tc}}\\
               G|_{\hat t_2\times\mathcal{F}_{tc}}
             \end{pmatrix}
             - \begin{pmatrix}
               Q_{t_1c'} E_{t_1c} Q_{tc}^* G_{tc}\\
               Q_{t_2c'} E_{t_2c} Q_{tc}^* G_{tc}
             \end{pmatrix} \right\|^2\\
  &= \|G|_{\hat t_1\times\mathcal{F}_{tc}}
       - Q_{t_1c'} E_{t_1c} Q_{tc}^* G_{tc}\|^2
     + \|G|_{\hat t_2\times\mathcal{F}_{tc}}
         - Q_{t_2c'} E_{t_2c} Q_{tc}^* G_{tc}\|^2\\
  &= \|G|_{\hat t_1\times\mathcal{F}_{tc}}
       - Q_{t_1c'} Q_{t_1c'}^* G|_{\hat t_1\times\mathcal{F}_{tc}}
       + Q_{t_1c'} (Q_{t_1c'}^* G|_{\hat t_1\times\mathcal{F}_{tc}}
                  - E_{t_1c} Q_{tc}^* G_{tc})\|^2\\
  &\quad + \|G|_{\hat t_2\times\mathcal{F}_{tc}}
       - Q_{t_2c'} Q_{t_2c'}^* G|_{\hat t_2\times\mathcal{F}_{tc}}
       + Q_{t_2c'} (Q_{t_2c'}^* G|_{\hat t_2\times\mathcal{F}_{tc}}
                  - E_{t_2c} Q_{tc}^* G_{tc})\|^2\\
  &= \|G|_{\hat t_1\times\mathcal{F}_{tc}}
         - Q_{t_1c'} Q_{t_1c'}^* G|_{\hat t_1\times\mathcal{F}_{tc}}\|^2
     + \|Q_{t_1c'}(Q_{t_1c'}^* G|_{\hat t_1\times\mathcal{F}_{tc}}
              - E_{t_1c} Q_{tc}^* G_{tc}\|^2\\
  &\quad + \|G|_{\hat t_2\times\mathcal{F}_{tc}}
         - Q_{t_2c'} Q_{t_2c'}^* G|_{\hat t_2\times\mathcal{F}_{tc}}\|^2
     + \|Q_{t_2c'}(Q_{t_2c'}^* G|_{\hat t_2\times\mathcal{F}_{tc}}
              - E_{t_2c} Q_{tc}^* G_{tc}\|^2\\
  &= \|G|_{\hat t_1\times\mathcal{F}_{tc}}
         - Q_{t_1c'} Q_{t_1c'}^* G|_{\hat t_1\times\mathcal{F}_{tc}}\|^2
   + \|G|_{\hat t_2\times\mathcal{F}_{tc}}
         - Q_{t_2c'} Q_{t_2c'}^* G|_{\hat t_2\times\mathcal{F}_{tc}}\|^2\\
  &\quad + \left\| \begin{pmatrix}
         Q_{t_1c'}^* G|_{\hat t_1\times\mathcal{F}_{tc}}\\
         Q_{t_2c'}^* G|_{\hat t_2\times\mathcal{F}_{tc}}
       \end{pmatrix}
       - \begin{pmatrix}
         E_{t_1c}\\ E_{t_2c}
       \end{pmatrix}
       \begin{pmatrix}
         E_{t_1c}\\ E_{t_2c}
       \end{pmatrix}^*
       \begin{pmatrix}
         Q_{t_1c'}^* G|_{\hat t_1\times\mathcal{F}_{tc}}\\
         Q_{t_2c'}^* G|_{\hat t_2\times\mathcal{F}_{tc}}
       \end{pmatrix} \right\|^2.
\end{align*}
We introduce the auxiliary matrices
\begin{align}\label{eq:hatQ_hatG}
  \widehat{Q}_{tc}
  &:= \begin{pmatrix} E_{t_1c}\\ E_{t_2c} \end{pmatrix}, &
  \widehat{G}_{tc}
  &:= \begin{pmatrix}
        Q_{t_1c'}^* G|_{\hat t_1\times\mathcal{F}_{tc}}\\
        Q_{t_2c'}^* G|_{\hat t_2\times\mathcal{F}_{tc}}
      \end{pmatrix}
\end{align}
and obtain
\begin{align}
  \|G_{tc} - Q_{tc} Q_{tc}^* G_{tc}\|^2
  &= \|G|_{\hat t_1\times\mathcal{F}_{tc}}
        - Q_{t_1c'} Q_{t_1c'}^* G|_{\hat t_1\times\mathcal{F}_{tc}}\|^2\notag\\
  &\quad+ \|G|_{\hat t_2\times\mathcal{F}_{tc}}
        - Q_{t_2c'} Q_{t_2c'}^* G|_{\hat t_2\times\mathcal{F}_{tc}}\|^2\notag\\
  &\quad+ \|\widehat{G}_{tc} - \widehat{Q}_{tc} \widehat{Q}_{tc}^*
          \widehat{G}_{tc}\|^2.\label{eq:error_nonleaf}
\end{align}
We can see that the first two terms on the right-hand side of the
equation only depend on $G$ and the already fixed matrices
$Q_{t_1c'}$ and $Q_{t_2c'}$.
If we assume that these matrices have been chosen appropriately,
we only have to find $\widehat{Q}_{tc}$ such that
\begin{equation}\label{eq:approx_hatGtc}
  \|\widehat{G}_{tc} - \widehat{Q}_{tc} \widehat{Q}_{tc}^*
       \widehat{G}_{tc}\|
  \leq \epsilon_t
\end{equation}
holds for a given accuracy $\epsilon_t>0$.
This problem can again be solved by computing the singular value
decomposition of $\widehat{G}_{tc}$ and using the first $k$ left
singular vectors as the columns of $\widehat{Q}_{tc}$.
Splitting the matrix according to (\ref{eq:hatQ_hatG}) gives us
the required transfer matrices $E_{t_1c}$ and $E_{t_2c}$.
Since $\widehat{G}_{tc}$ has only $2k$ rows, this procedure is
efficient as long as the rank $k$ is not too high.

In order to compute $\widehat{G}_{tc}$ efficiently, i.e., without
going back to the original matrix $G|_{\hat t\times\mathcal{F}_{tc}}$,
we introduce the auxiliary matrices
\begin{align*}
  R_{tc} &:= Q_{tc}^* G|_{\hat t\times\mathcal{F}_{tc}} &
  &\text{ for all } t\in\ctI,\ c\in\mathcal{D}_t
\end{align*}
and observe that $\mathcal{F}_{tc}\subseteq\mathcal{F}_{t_1,c'},
\mathcal{F}_{t_2,c'}$ yields
\begin{equation*}
  \widehat{G}_{tc}
  = \begin{pmatrix}
      R_{t_1c'}|_{k\times\mathcal{F}_{tc}}\\
      R_{t_2c'}|_{k\times\mathcal{F}_{tc}}
    \end{pmatrix},
\end{equation*}
i.e., we can construct $\widehat{G}_{tc}$ efficiently by copying
suitable submatrices of $R_{t_1c'}$ and $R_{t_2c'}$.
If $t$ is a leaf, we can compute $R_{tc}$ by definition, since we
can assume that $\hat t$ is small.
If $t$ is not a leaf, we can use
\begin{equation*}
  R_{tc}
  = Q_{tc}^* G|_{\hat t\times\mathcal{F}_{tc}}
  = \widehat{Q}_{tc}^*
    \begin{pmatrix}
      Q_{t_1c'}^* G|_{\hat t_1\times\mathcal{F}_{tc}}\\
      Q_{t_2c'}^* G|_{\hat t_2\times\mathcal{F}_{tc}}
    \end{pmatrix}
  = \widehat{Q}_{tc}^* \widehat{G}_{tc}
\end{equation*}
to obtain the matrix efficiently.
The resulting algorithm is summarized in Figure~\ref{fi:dh2basis}.

%
% Figure: DH2-compression
%
\begin{figure}
  \begin{tabbing}
    \textbf{procedure} basis($t$, $G$,
                          \textbf{var} $R=(R_{tc})_{t\in\ctI,c\in\mathcal{D}_t}$,
                                    $Q=(Q_{tc})_{t\in\ctI,c\in\mathcal{D}_t}$);\\
    \textbf{if} $\sons(t)=\emptyset$ \textbf{then begin}\\
    \quad\= \textbf{for} $c\in\mathcal{D}_t$ \textbf{do begin}\\
    \> \quad\= Construct $Q_{tc}$ from the first $k$ singular vectors
                  of $G_{tc}$;\\
    \> \> $R_{tc} \gets Q_{tc}^* G_{tc}$\\
    \> \textbf{end}\\
    \textbf{else begin}\\
    \> \textbf{for all} $t'\in\sons(t)$ \textbf{do}
        basis($t'$, $G$, $R$, $Q$);\\
    \> \textbf{for all} $c\in\mathcal{D}_t$ \textbf{do begin}\\
    \> \> $\widehat{G}_{tc} \gets
             \begin{pmatrix}
               R_{t_1c'}|_{k\times\mathcal{F}_{tc}}\\
               R_{t_2c'}|_{k\times\mathcal{F}_{tc}}
             \end{pmatrix}$;\\
    \> \> Construct $\widehat{Q}_{tc}$ from the first $k$ singular vectors
                  of $\widehat{G}_{tc}$;\\
    \> \> Recover $E_{t_1c}$ and $E_{t_2c}$ from $\widehat{Q}_{tc}$;\\
    \> \> $R_{tc} \gets \widehat{Q}_{tc}^* \widehat{G}_{tc}$\\
    \> \textbf{end}\\
    \textbf{end}
  \end{tabbing}
  \caption{Construction of an orthogonal directional cluster basis}
  \label{fi:dh2basis}
\end{figure}

%
% Theorem: Error estimate
%
\begin{theorem}[Error estimate]
\label{th:error_estimate}
If (\ref{eq:approx_Gtc}) holds for all leaf clusters $t\in\lfI$
and (\ref{eq:approx_hatGtc}) holds for all non-leaf clusters
$t\in\ctI\setminus\lfI$, we have
\begin{align*}
  \|G_{tc} - Q_{tc} Q_{tc}^* G_{tc}\|^2
  &\leq \sum_{r\in\desc(t)} \epsilon_r^2 &
  &\text{ for all } t\in\ctI,\ c\in\mathcal{D}_t.
\end{align*}
\end{theorem}
\begin{proof}
Structural induction using (\ref{eq:error_nonleaf}).
\end{proof}

%
% Remark: Error control
%
\begin{remark}[Error control]
\label{re:error_control}
The result of Theorem~\ref{th:error_estimate} holds for individual
submatrices:
for $b=(t,s)\in\lfaII$, we find
\begin{equation*}
  \|G|_{\hat t\times\hat s} - Q_{tc} Q_{tc}^* G|_{\hat t\times\hat s}\|^2
  \leq \sum_{s\in\desc(t)} \epsilon_r^2.
\end{equation*}
Choosing $\epsilon_r \sim \zeta^{\level(r)-\level(t)}$ with
$\zeta < \sqrt{1/\Csn}$ (cf. (\ref{eq:sons_bound})) turns
the right-hand side into a geometric sum that can be bounded
independently of $t$ and $s$.

Refined error control techniques \cite{BO05a} can be implemented
by weighting the submatrices:
let $(\omega_{ts})_{t\in\ctI,s\in F_{tc}}$ be a family of weights
$\omega_{ts}\in\bbbr_{>0}$ and define $G^\omega_{tc}$ by
\begin{align*}
  G^\omega_{tc}|_{\hat t\times\hat s}
  &= \omega_{ts}^{-1} G|_{\hat t\times\hat s} &
  &\text{ for all } t\in\ctI,\ c\in\mathcal{D}_t,\ s\in F_{tc}.
\end{align*}
Replacing $G_{tc}$ in (\ref{eq:approx_Gtc}) and (\ref{eq:hatQ_hatG})
by $G^\omega_{tc}$ leads to
\begin{align*}
  \omega_{ts}^{-2}
   \|G|_{\hat t\times\hat s} - Q_{tc} Q_{tc}^* G|_{\hat t\times\hat s}\|^2
  &\leq \phantom{\omega_{ts}^2} \sum_{r\in\desc(t)} \epsilon_r^2 &
  &\iff\\
  \|G|_{\hat t\times\hat s} - Q_{tc} Q_{tc}^* G|_{\hat t\times\hat s}\|^2
  &\leq \omega_{ts}^2 \sum_{r\in\desc(t)} \epsilon_r^2 &
  &\text{ for all } b=(t,s)\in\lfaII,
\end{align*}
so we can choose different accuracies for each block, e.g., to
ensure block-relative error bounds by using
$\omega_{ts} = \|G|_{\hat t\times\hat s}\|$.
\end{remark}

In order to obtain an estimate for the complexity of the
compression algorithm, we assume that there is a constant
$\Csvd\in\bbbr_{>0}$ such that the singular value decomposition
of a $n$-by-$m$ matrix can be computed in not more than
\begin{equation}\label{eq:work_svd}
  \Csvd \min\{n^2,m^2\} \max\{n,m\}
  \text{ operations}
\end{equation}
up to machine accuracy (cf. \cite[Section~5.4.5]{GOVL96}).

%
% Theorem: Complexity
%
\begin{theorem}[Complexity]
\label{th:complexity_compression}
If (\ref{eq:rank_leaves}) and (\ref{eq:work_svd}) hold, the
compression algorithm given in Figure~\ref{fi:dh2basis} requires
not more than
\begin{equation*}
  \Cba k^2 (\#\ctI) (\#\Idx) \text{ operations},
\end{equation*}
where $\Cba := \max\{ \Csvd \Crs^2 + 2 \Crs,
4 \Csvd + 4 \}$.

If also (\ref{eq:sons_bound}) holds, we find that we require not
more than
\begin{equation*}
  \Cba \Clv k (\#\Idx)^2 \text{ operations}.
\end{equation*}
\end{theorem}
\begin{proof}
Let $t\in\ctI$.

If $t$ is a leaf, the algorithm computes the singular value
decomposition of $G_{tc}$.
This matrix has $\#\hat t$ rows and $\#\mathcal{F}_{tc}$ columns,
so (\ref{eq:work_svd}) yields that not more than
\begin{equation*}
  \Csvd (\#\hat t)^2 \#\mathcal{F}_{tc}
  \leq \Csvd \Crs^2 k^2 \#\mathcal{F}_{tc}
  \text{ operations}    
\end{equation*}
are required to find $Q_{tc}$.
The multiplication needed to compute $R_{tc}$ takes not more than
\begin{equation*}
  2 k (\#\hat t) \#\mathcal{F}_{tc}
  \leq 2 \Crs k^2 \#\mathcal{F}_{tc}
  \text{ operations}.
\end{equation*}
Since the sets $\mathcal{F}_{tc}$ for different directions
$c\in\mathcal{D}_t$ are disjoint, we have a total of not more than
\begin{align*}
  \sum_{c\in\mathcal{D}_t}
     (\Csvd \Crs^2 + 2 \Crs) k^2 \#\mathcal{F}_{tc}
  &= (\Csvd \Crs^2 + 2 \Crs) k^2
          \#\bigcup_{c\in\mathcal{D}_t} \mathcal{F}_{tc}\\
  &\leq \Cba k^2 \#\Idx \text{ operations}
\end{align*}
for a leaf cluster.

If $t$ is not a leaf, forming the matrix $\widehat{G}_{tc}$ requires
no arithmetic operations, computing its singular value decomposition
requires not more than
\begin{equation*}
  \Csvd (2k)^2 \#\mathcal{F}_{tc}
  = 4 \Csvd k^2 \#\mathcal{F}_{tc}
  \text{ operations},
\end{equation*}
copying the contents of $\widehat{Q}_{tc}$ into the transfer matrices
again takes no arithmetic operations, and computing $R_{tc}$ requires
not more than
\begin{equation*}
  2 k (2k) \#\mathcal{F}_{tc}
  = 4 k^2 \#\mathcal{F}_{tc}
  \text{ operations}.
\end{equation*}
Since the sets $\mathcal{F}_{tc}$ for different $c\in\mathcal{D}_t$
are disjoint, we can use the same argument as before to conclude that
\begin{equation*}
  (4 \Csvd+4) k^2 \#\Idx
  \leq \Cba k^2 \#\Idx
  \text{ operations}
\end{equation*}
are sufficient for a non-leaf cluster.

Adding up the estimates for all clusters yields the first estimate,
and the estimate (\ref{eq:cluster_bound}) of Lemma~\ref{le:clusters}
yields the second.
\end{proof}

% ============================================================
%
% Numerical experiments
%
% ============================================================
\section{Numerical experiments}

We first consider the approximation of the single layer matrix
of the Helmholtz integral operator on the unit sphere.
The surface mesh is constructed by taking the double pyramid
$\{ x\in\bbbr^3\ :\ |x_1|+|x_2|+|x_3|=1 \}$, regularly refining
its eight triangular sides, and shifting the vertices of the
resulting mesh to the unit sphere $\{ x\in\bbbr^3\ :\ \|x\|=1 \}$.
The Galerkin stiffness matrix $G\in\bbbc^{\Idx\times\Idx}$ for
piecewise constant basis functions is approximated by
Sauter-Erichsen-Schwab quadrature \cite{ERSA98,SASC11} using
$3$ quadrature points per coordinate direction.

The cluster tree is constructed by standard geometrically regular
subdivision stopping at leaf clusters containing at most $16$
indices.
The directions are constructed by the procedure described in
Section~\ref{se:admissibility} with $\eta_1=20$.
The block tree is constructed using the admissibility conditions
(\ref{eq:adm_parabolic}) and (\ref{eq:adm_standard}) with $\eta_2=5$.

We apply the compression algorithm with the block-relative
error tolerance
\begin{align*}
  \| G|_{\hat t\times\hat s}
     - Q_{tc} Q_{tc}^* G|_{\hat t\times\hat s} \|
  &\leq \frac{\epsilon}{1-2\zeta^2} \| G|_{\hat t\times\hat s} \| &
  &\text{ for all } b=(t,s)\in\lfaII,
\end{align*}
where we choose $\zeta=2/3$ (cf. Remark~\ref{re:error_control})
and $\epsilon=10^{-4}$.

%
% Table: Directional compression, SLP
%
\begin{table}
  \begin{equation*}
  \begin{array}{rr|rrrrrrr}
    n & \kappa & t_\text{row} & t_\text{col} & t_\text{prj}
               & k_\text{max} & \text{Mem}/n & t_\text{mvm}
               &\frac{\|G-\widetilde G\|_2}{\|G\|_2}\\
    \hline
    2048 & 8 & 0.6 & 0.4 & 0.6 & 19 & 24.2 & 0.04 & 6.4_{-6}\\
    4608 & 12 & 2.4 & 1.8 & 1.2 & 26 & 44.6 & 0.2 & 5.7_{-6}\\
    8192 & 16 & 6.6 & 5.2 & 4.5 & 29 & 61.4 & 0.4 & 7.3_{-6}\\
   18432 & 24 & 26.5 & 22.1 & 15.4 & 34 & 83.2 & 1.3 & 7.3_{-6}\\
   32768 & 32 & 95.3 & 79.5 & 47.0 & 38 & 93.5 & 3.3 & 8.0_{-6}\\
   73728 & 48 & 554.6 & 509.6 & 236.2 & 38 & 97.0 & 6.9 & 9.2_{-6}\\
  131072 & 64 & 1696.6 & 1585.4 & 803.4 & 41 & 102.1 & 10.7 & 7.5_{-6}\\
  294912 & 96 & 6704.7 & 6743.2 & 3902.5 & 46 & 115.8 & 30.1 & 8.2_{-6}
  \end{array}
  \end{equation*}

  \caption{$\mathcal{DH}^2$-matrix compression of the single-layer potential,
           $\eta_1=20$, $\eta_2=5$, $\epsilon=10^{-4}$}
  \label{ta:dh2slp_dblock}
\end{table}

%
% Table: Directional compression, DLP
%
\begin{table}
  \begin{equation*}
  \begin{array}{rr|rrrrrrr}
    n & \kappa & t_\text{row} & t_\text{col} & t_\text{prj}
               & k_\text{max} & \text{Mem}/n & t_\text{mvm}
               &\frac{\|G-\widetilde G\|_2}{\|G\|_2}\\
    \hline
    2048 & 8 & 0.5 & 0.4 & 0.6 & 22 & 24.9 & 0.04 & 8.8_{-6}\\
    4608 & 12 & 2.3 & 1.9 & 1.3 & 29 & 46.6 & 0.2 & 8.1_{-6}\\
    8192 & 16 & 6.5 & 5.3 & 3.8 & 33 & 65.4 & 0.5 & 1.0_{-5}\\
   18432 & 24 & 29.6 & 23.3 & 17.1 & 38 & 88.5 & 1.6 & 1.3_{-5}\\
   32768 & 32 & 81.9 & 84.5 & 48.9 & 41 & 100.3 & 2.9 & 1.4_{-5}\\
   73728 & 48 & 490.4 & 508.7 & 260.3 & 38 & 102.3 & 7.0 & 1.7_{-5}\\
  131072 & 64 & 1576.6 & 1588.5 & 850.0 & 42 & 107.9 & 15.1 & 1.5_{-5}\\
  294912 & 96 & 6697.4 & 7263.9 & 4233.1 & 46 & 121.4 & 33.8 & 1.9_{-5}
  \end{array}
  \end{equation*}

  \caption{$\mathcal{DH}^2$-matrix compression of the double-layer potential,
           $\eta_1=20$, $\eta_2=5$, $\epsilon=10^{-4}$}
  \label{ta:dh2dlp_dblock}
\end{table}

\paragraph*{SLP matrix.}
The results of our experiment are collected in Table~\ref{ta:dh2slp_dblock}.
Its first column contains the matrix dimension $n$, the second the
wave number $\kappa$.
The wave number has been chosen such that $\kappa h \approx 1.3$, i.e., we
are in the high-frequency regime with only approximately five mesh
elements per wavelength.

The third, fourth and fifth column give the time in seconds required
to construct the directional row basis, the directional column basis,
and the final $\mathcal{DH}^2$-matrix approximation.
The implementation is parallelized based on a decomposition of the
cluster tree into independent subtrees.
The program was allowed to use up to $64$ cores of a SGI~UV2000 shared
memory computer with Intel Xeon E5-4640 processors running at 2.4~GHz.

The sixth column gives the maximal rank $k$ used in the adaptively
constructed bases, and the storage requirements in KiB per degree
of freedom can be found in the seventh column.
The eigth column gives the time in seconds required for a matrix-vector
multiplication by the $\mathcal{DH}^2$-matrix $\widetilde{G}$.
This operation is currently only partially parallelized: the forward
and backward transformation are performed sequentially, while the
coupling and nearfield matrices are handled concurrently with up
to $64$ cores.
The ninth and last column contains the relative spectral error measured
by a power iteration.

We first notice that the error control strategy works even better
than expected: the relative spectral errors are approximately ten times
smaller than the prescribed error tolerance $\epsilon$.

We can also see that the rank appears to grow like $\log(\kappa)$.
This effect can be explained by applying \cite[Lemma~6.37]{BO10}:
the approximation of $\widehat{G}_{tc}$ depends on the number
of clusters in $F_{tc}$, and Lemma~\ref{le:sparsity} suggests
that this number grows like $\kappa^2$.
The compression algorithm chooses a higher rank, corresponding
to a higher expansion order, to compensate this growth.

%
% Figure: Relative time and storage
%
\begin{figure}
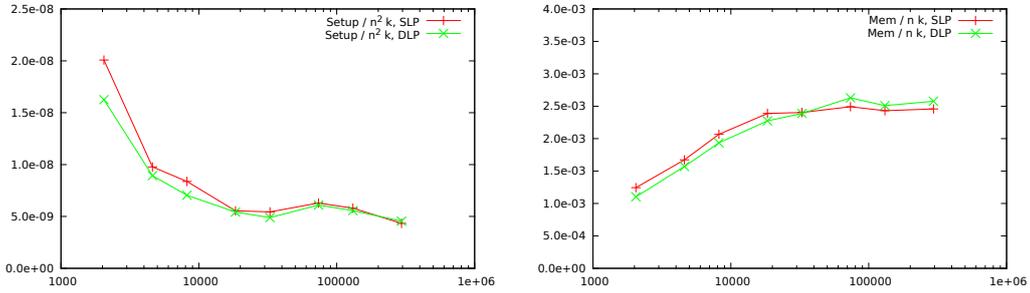

  \pgfdeclareimage[width=0.45\textwidth]{setup}{fi_setup}
  \pgfdeclareimage[width=0.45\textwidth]{mem}{fi_mem}

  \begin{center}
  \pgfuseimage{setup}%
  \quad%
  \pgfuseimage{mem}
  \end{center}

  \caption{Run-time for the $\mathcal{DH}^2$-compression relative to
           $n^2 k$ (left)
           and storage requirements relative to $n k$ (right)}
  \label{fi:setup_mem}
\end{figure}

Theorem~\ref{th:complexity_compression} predicts that the number
of operations for finding an adaptive directional cluster basis
is bounded by $n^2 k$.
The time divided by $n^2 k$ is displayed on the left in
Figure~\ref{fi:setup_mem}, and we can see that it indeed appears
to be bounded.

Theorem~\ref{th:complexity} predicts that the storage requirements
of the $\mathcal{DH}^2$-matrix approximation are bounded by
$\mathcal{O}(n k^2 \log n)$.
The right-hand side of Figure~\ref{fi:setup_mem} shows the storage
requirements divided by $n k$, and since the curve appears to be
bounded, our complexity estimate may be pessimistic.
A possible explanation could be that the estimate considers only the
\emph{maximal} rank $k$, while the complexity can be expected to depend
on a weighted \emph{average} rank.

\paragraph*{DLP matrix.}
Since the compression algorithm requires only the block structure
and the matrix coefficients, we can apply it to more general matrices
and determine experimentally whether they can be represented
efficiently in the $\mathcal{DH}^2$-matrix format.

We first consider the Helmholtz double layer potential (DLP) operator
given by the kernel function
\begin{equation*}
  g_\text{dlp}(x,y) = \frac{\partial}{\partial n(y)} g(x,y)
  = (1 - i \kappa \|x-y\|) \frac{\exp(i \kappa \|x-y\|)}{4 \pi \|x-y\|^3}
    \langle x-y, n(y) \rangle.
\end{equation*}
We denote the resulting Galerkin matrix by $G_\text{dlp}$.
Since it usually appears in second-kind integral equations, we
approximate $G := \frac{1}{2} M + G_\text{dlp}$, where $M$ denotes the
mass matrix.
The results are given in Table~\ref{ta:dh2dlp_dblock}, and we can
see that the compression algorithm works as well for the double layer
potential as for the single layer potential.
We have included the runtime for the compression and the
storage requirements in Figure~\ref{fi:setup_mem}, and observe that
the curves for SLP and DLP look very similar.

\paragraph*{Comparison with ACA.}
Practical experiments show that the \emph{adaptive cross approximation}
(ACA) method \cite{BE00a} works surprisingly well for the Helmholtz
boundary element method, even in the case of fairly high frequencies.
In a final experiment, we compare the new compression algorithm
to ACA for the single layer potential operator.
Since ACA uses the standard admissibility condition
\begin{equation*}
  \max\{\diam(B_t), \diam(B_s)\} \leq \eta_2 \dist(B_t,B_s)
\end{equation*}
instead of the parabolic condition (\ref{eq:admissibility}), we level
the playing field and use the same condition also for the
$\mathcal{DH}^2$-matrix compression algorithm.

%
% Table: DH2 vs ACA, SLP
%
\begin{table}
  \begin{equation*}
  \begin{array}{rr|rrr|rrr}
      &        & \multicolumn{3}{c|}{\text{ACA}}
               & \multicolumn{3}{c}{\mathcal{DH}^2}\\
    n & \kappa & k_\mathrm{max} & \text{Mem}/n
               & \frac{\|G-\widetilde G\|_2}{\|G\|_2}
               & k_\mathrm{max} & \text{Mem}/n
               & \frac{\|G-\widetilde G\|_2}{\|G\|_2}\\
    \hline
    8192 & 16 & 40 & 30.0 & 8.5_{-5}
              & 49 & 15.7 & 4.1_{-5}\\
   18432 & 24 & 54 & 40.1 & 1.2_{-4}
              & 72 & 18.0 & 4.7_{-5}\\
   32768 & 32 & 71 & 50.0 & 1.9_{-4}
              & 71 & 19.9 & 5.0_{-5}\\
   73728 & 48 & 102 & 67.3 & 2.4_{-4}
              & 104 & 21.9 & 5.3_{-5}\\
  131072 & 64 & 143 & 86.1 & 2.8_{-4}
              & 147 & 23.7 & 5.1_{-5}\\
  \end{array}
  \end{equation*}

  \caption{Adaptive cross approximation compared to $\mathcal{DH}^2$-matrix
           compression for the single-layer potential operator with
           the standard admissibility condition, $\eta_2=5$, $\epsilon=10^{-4}$}
  \label{ta:dh2_vs_aca}
\end{table}

The results are given in Table~\ref{ta:dh2_vs_aca}.
We can see that the $\mathcal{DH}^2$-matrix compression requires
a significantly smaller amount of storage while yielding a higher
accuracy.
The advantage of the $\mathcal{DH}^2$-matrix grows as the problems
increase in size.

\bibliographystyle{plain}
\bibliography{hmatrix}

\end{document}